\documentclass{paper}

\usepackage{arxiv}

\usepackage[utf8]{inputenc} 
\usepackage[T1]{fontenc}    
\usepackage{hyperref}       
\usepackage{url}            
\usepackage{booktabs}       
\usepackage{amsfonts}       
\usepackage{nicefrac}       
\usepackage{microtype}      
\usepackage{subcaption}
\usepackage{mathrsfs,bm,siunitx}
\usepackage{comment}
\usepackage{appendix}
\usepackage{lipsum}
\usepackage{graphicx}
\usepackage{amsmath,amsthm}
\numberwithin{equation}{section}
\usepackage{color}
\usepackage{colortbl}
\usepackage{xcolor}
\usepackage{transparent}
\usepackage{todonotes}

\graphicspath{{}}
\usepackage{tikz}
\usetikzlibrary{arrows}
\newtheorem{theorem}{Theorem}[section]
\newtheorem{proposition}{Proposition}[section]
\newtheorem{lemma}{Lemma}[section]
\newtheorem{corollary}{Corollary}[section]
\theoremstyle{definition}
\newtheorem{dfn}{Definition}[section]
\newtheorem{remark}{Remark}[section]

\DeclareMathOperator*{\esssup}{ess\,sup}
\DeclareMathOperator{\sgn}{sgn}

\title{Solving an inverse heat convection problem\\ with an implicit forward operator by using\\ a Projected Quasi-Newton method}

\author{
  Dimitri Rothermel \\
  Department of Numerical Mathematics\\
  Saarland University\\
  Saarbr\"ucken, Germany \\
  \texttt{dimitri.rothermel@num.uni-sb.de} \\
       \And
  Thomas Schuster \\
  Department of Numerical Mathematics\\
  Saarland University\\
  Saarbr\"ucken, Germany \\
  \texttt{thomas.schuster@num.uni-sb.de} \\
}

\begin{document}
\maketitle

\begin{abstract}
We consider the quasilinear 1D inverse heat convection problem (IHCP) of determining the enthalpy-dependent heat fluxes from noisy internal enthalpy measurements. This problem arises in the Accelerated Cooling (ACC) process of producing thermomechanically controlled processed (TMCP) heavy plates made of steel. In order to adjust the complex microstructure of the underlying material, the Leidenfrost behavior of the hot surfaces with respect to the application of the cooling fluid has to be studied. Since the heat fluxes depend on the enthalpy and hence on the solution of the underlying initial boundary value problem (IBVP), the parameter-to-solution operator, and thus the forward operator of the inverse problem, can only be defined implicitly. To guarantee well-defined operators, we study two approaches for showing existence and uniqueness of solutions of the IBVP. One approach deals with the theory of pseudomonotone operators and so-called strong solutions in Sobolev-Bochner spaces. The other theory uses classical solutions in H\"older spaces. Whereas the first approach yields a solution under milder assumptions, it fails to show the uniqueness result in contrast to the second approach. Furthermore, we propose a convenient parametrization approach for the nonlinear heat fluxes  in order to decouple the parameter-to-solution relation and use an iterative solver based on a Projected Quasi-Newton (PQN) method together with box-constraints to solve the inverse problem. For numerical experiments, we derive the necessary gradient information of the objective functional and use the discrepancy principle as a stopping rule. Numerical tests show that the PQN method outperforms the Landweber method with respect to computing time and approximation accuracy.
\end{abstract}


\smallskip


\smallskip

\section{Introduction}\label{sec:intro}
In this paper, we investigate a heat transfer model given by the initial boundary value problem (IBVP) for a 1D quasilinear parabolic second order equation
\begin{align}
    \label{IBVP1}
     u_t &= (\alpha'(u)u_x)_x,  &t \in I,\  x\in \Omega,\\
     \label{IBVP2}
    \alpha'(u)u_x &= \beta_0(u), &t \in I,\  x=0,\\
    \label{IBVP3}
    -\alpha'(u)u_x &= \beta_L(u), &t \in I,\  x=L,\\
    \label{IBVP4}
    u&=u_0, &t=0, x\in \bar{\Omega},
\end{align}
where $t\in I:=[0,T]$ for $T>0$ denotes the time and $x\in \Omega:=(0,L)$ for $L>0$ the space variable. The solution $u=u(t,x)$ of the IBVP is the so-called enthalpy. The subscripts $t$ and $x$ refer to the respective derivative, $u_0$ is the initial enthalpy distribution and the heat fluxes $\beta_0,\beta_L>0$ are enthalpy-dependent. For all $(t,x)\in I\times \bar{\Omega}$ the term $\alpha'(u(t,x))$ represents the thermal diffusivity of the underlying material, which is typically a function of temperature rather than enthalpy, see \eqref{alphaStrichGleichTLF}.

The IBVP is motivated by the modeling of the Accelerated Cooling (ACC) process of thermomechanically controlled processed (TMCP) heavy plates made of steel. In the TMCP process hot-rolled heavy plates are subjected to a strong cooling by applying water on both, bottom and top surface, which we refer to as $x=0$ and $x=L$, respectively. This leads to a complex crystalline microstructure of the material affecting the mechanical properties such as toughness, strength and weldability, cf. \cite{laepple}. By understanding this interaction one can vary the cooling plan, in order to adjust these steel properties in a desired manner. Economically, this is very efficient due to the saving of expensive alloying agents. Furthermore, the ability to control the cooling process is of great interest in minimizing residual stresses in the steel plates.

We mention, that our heat transfer model (\eqref{IBVP1} - \eqref{IBVP4}) is 1D, whereas the heavy plates are obviously 3D. The dimension reduction arises from the fact, that the plates are very thin (x-direction) compared to their length and height ($y$-$z$-plane). Thus, the heat transfer is mainly governed over the thickness dimension $x\in \bar{\Omega}.$ For example, at depth $x_c=L/2$ and times $t\in I$, we assume that $u(t,x_c,y_1,z_1)=u(t,x_c,y_2,z_2)$ for all tuples $(y_1,z_1)$ and $(y_2,z_2)$ within the plate. In this way we can neglect the $y$-$z$-plane and just refer to $u(t,x_c)$ for the core enthalpy of the heavy plate.

The enthalpy $u: I \times \Omega \to \mathbb{R}_+$ is defined by the temperature-to-enthalpy transformation
\begin{align}
    \label{enthalpyDef}
u(t,x):=\hat{C}(\theta(t,x)):=\int\limits_{\underline{\theta}}^{\theta(t,x)}C(\xi)\ \mathrm{d}\xi>0, 
\end{align}
where $\theta>0$ is the temperature variable, $\underline{\theta}=273.15\  K(=0^\circ \ \text{Celsius})$ a constant and $C:\mathbb{R}_+\to \mathbb{R}_+$ the temperature-dependent volumetric heat capacity, i.e. a product of the mass density and the specific heat capacity of the underlying material. Note, that we often omit the variables $(t,x)$ for a better readability. Since $\hat{C}'=C(\theta)>0$, the inverse mapping $\hat{C}^{-1}$ exists and the enthalpy-to-temperature mapping is simply given by $\theta=\hat{C}^{-1}(u)$. As already discussed in \cite{rothermel2019parameter}, the heat conduction behavior of hot-rolled heavy plates during the Accelerated Cooling (ACC) process is determined by strong temperature gradients and phase transitions in the underlying material, hence leading to the 1D heat transfer model describing the temperature distribution with respect to the thickness of the plate by 
\begin{align}
    \label{oldPde}
    C(\theta)\theta_t = (k(\theta)\theta_x)_x \ \ \ \  \text{for } (t,x) \in I\times\Omega.
\end{align}
Equation \eqref{oldPde} is often used in the modeling of dynamic heating processes, too, i.e. when adding heat to the system, e.g. in simulations of blast furnaces or in space research. The equation is governed by temperature dependent material parameters $C$ and $k$, where $k:\mathbb{R}_+\to \mathbb{R}_+$ is the so-called thermal conductivity. 

By using the transformation approach from \cite{roubicek}, when defining
\begin{align}
    \label{enthalpyDiffusion}
    \alpha(u):&= \int\limits_{\underline{\theta}}^{\hat{C}^{-1}(u)} k(\xi)\ \mathrm{d}\xi, \ \ \text{such that}\\
    \alpha'(u)&=\frac{k(\theta)}{C(\theta)} \label{alphaStrichGleichTLF},
\end{align}
the equations \eqref{IBVP1} and \eqref{oldPde} describe the same heat transfer process since $ u_t = C(\theta)\theta_t$ and $u_x = C(\theta)\theta_x$.
This way, pointwise for any $(t,x)$, we see that $\alpha'(u(t,x))$ represents the so-called thermal diffusivity $\frac{k(\theta(t,x))}{C(\theta(t,x))}$, which is a bounded function typically stated only in terms of temperature $\theta$. 

We want to point out the advantage of describing the heat transfer for both formulations. First of all, in industrial applications the end user often thinks in terms of temperature, while the enthalpy often seems intangible. For instance, the insertion of thermocouples into the heavy plates (cf. \cite{rothermel2019parameter}) typically measure temperatures. On the other hand, the enthalpy formulation fits the approach of \textit{pseudomonotone} operators in order to show that the involved PDEs are meaningful in an abstract mathematical sense. This is why we focus on the heat transfer process described in terms of enthalpy in this paper.

The motivation of considering enthalpy (or temperature) dependent heat fluxes $\beta_0,\beta_L>0$ in  \eqref{IBVP2} - \eqref{IBVP3} is to model the so-called Leidenfrost effect, cf. \cite{walker2010boiling}, \cite{bernardin1999leidenfrost}. The hot surfaces form a steam layer and avoid contact with the cooling fluid until the steam layer collapses in some critical enthalpy (or temperature) range. Because of different cooling conditions on the bottom ($x=0$, e.g. insulation due to the roller contact) and top surface ($x=L$, e.g. accumulation of backwater) it is appropriate to consider not only one, but two heat fluxes $\beta_0$ and $\beta_L$.

Assuming that $\alpha'$ and $u_0$ is known, the aim of this paper, is to solve the inverse heat convection problem (IHCP) in determining the unknown heat fluxes $\beta_0$ and $\beta_L$ from noisy internal enthalpy measurements $u^{\delta}$. Since the heat fluxes depend on the solution u of the IBVP itself, a classical formulation of the inverse problem is not possible, because the forward operator $F$ could only be posed implicitly, e.g., with an appropriate definition, by
\begin{align}
F(\beta_0(u),\beta_L(u),u)=0.
\end{align}
In a classical inverse problem, we typically regard the forward operator as a composition of two operators, i.e.
 \begin{align}
     \label{compositionForF}
     F = \mathcal{Q}\circ S,
 \end{align}
where, $S$ is the explicitly defined parameter-to-solution operator, and  $\mathcal{Q}$ is the so-called observation operator, which encodes the measuring process and makes solutions $u$ of the IBVP comparable to the measurement data. In our case, however, we are not able to decouple the parameter-to-solution relation. Again, an operator $S$ could only be formulated implicitly, cf. Section \ref{subsec:parametrizationApproach}, where we propose a convenient parametrization approach in order to avoid this problem. We achieve this by introducing a piecewise cubic interpolation method, such that we can represent positive heat fluxes from $\mathcal{C}^1([0,u_{max}])$ by a heat flux parameter $\pmb{\beta}\in B:=[0,\beta_{max}]^{2n}\subset\mathbb{R}^{2n}_+$, i.e. $\beta_0(u)=\beta_0(u,\pmb{\beta})$ and $\beta_L(u)=\beta_L(u,\pmb{\beta})$. Here, $u_{max},\beta_{max}>0$ are known finite upper bounds.

Hence, we can explicitly define the parameter-to-solution operator by
\begin{align}
    \label{parameterToSolutionMap}
    S:B\subset\mathbb{R}^{2n} & \to \mathcal{U},\\
    \pmb{\beta} & \mapsto u,
\end{align}
and obtain to the nonlinear forward operator
 \begin{align}
     \label{forwardOperatorInIntroduction}
     F:B\subset \mathbb{R}^{2n} & \to Y,\\
     \pmb{\beta} & \mapsto u_s,
 \end{align}
which maps the heat flux parameter to some observed enthalpy state $u_s:=\mathcal{Q}u\in Y$ of the solution $u$ of the IBVP, where we insert the associated heat fluxes $\beta_0(u,\pmb{\beta}),\beta_L(u,\pmb{\beta})>0$ in \eqref{IBVP2} - \eqref{IBVP3}. 
 
As one of the main contributions of the paper, before studying the IHCP, we scrutinize the parameter-to-solution operator with regard to existence and uniqueness aspects of the IBVP \eqref{IBVP1} - \eqref{IBVP4} in an appropriate space $\mathcal{U}.$

The IHCP can then be posed by the nonlinear operator equation
\begin{align}
    \label{operatorEquation}
    F(\pmb{\beta})=u^{\delta}.
\end{align}
Note, that only the noisy measurement $u^{\delta}$ of the exact enthalpy state $u_s$ is available, where we assume that the noise level $\delta>0$ fulfills
\begin{align}
    \label{noiseLevel}
    \|u^{\delta}-u_s\|_Y\leq \delta.
\end{align}
Thus, since the solution of such operator equations usually does not depend continuously on the data, a direct inversion leads to no useful result. The inverse problem is ill-posed and requires the use of some regularization technique, see e.g. \cite{louis2013inverse}, \cite{engl1996regularization}, \cite{schuster}. One way to find a stable solution is to implement an iterative method to minimize the objective functional
\begin{align}
    \label{objFunctional}
    f(\pmb{\beta}):=\frac{1}{2}\|F(\pmb{\beta})-u^{\delta}\|_Y^2.
\end{align}
To solve the inverse problem, given an initial guess $\pmb{\beta}^{(0)}$, one performs some iteration procedure
\begin{align}
    \label{iterationRule}
    \pmb{\beta}^{(k+1)}=\pmb{\beta}^{(k)}+\lambda_k \mathbf{p}_k,
\end{align}
where $\mathbf{p}_k$ is the descent direction and $\lambda_k$ the step size. As discussed, e.g. in \cite{kaltenbacher2008iterative} for the Landweber method or in \cite{wang2005convergence} for Trust Region methods, a convergence can only be expected if the iteration \eqref{iterationRule} is stopped appropriately for example with Morozov's a-posteriori discrepancy principle, see \cite{morozov1966solution}. 

In this paper, we iteratively  solve
\begin{align}
    \label{minimizationTask}
    \min\limits_{\pmb{\beta}\in \mathbb{R}^{2n}} f(\pmb{\beta}) \text{ such that } \pmb{\beta}\in B
\end{align}
together with the discrepancy principle as a stopping rule, by implementing the Projected Quasi-Newton Method (PQN, proposed by \cite{kim2010tackling}), which is a gradient-scaling method that approximates the Hessian by Broyden-Fletcher-Goldfarb-Shanno (BFGS) updates. In comparison to the slow but very stable attenuated Landweber method, the PQN method excels by its convergence speed while still being robust. 

An essential information to compute some descent direction $\mathbf{p}_k$ is the gradient of the objective functional. Assuming that $F$ is differentiable with Fréchet derivative F', the gradient is explicitly given by
\begin{align}
    \label{gradientOfObjectiveFunctional}
    \nabla f(\pmb{\beta})=F'(\pmb{\beta})^*(F(\pmb{\beta})-u^{\delta}),
\end{align}
cf. \cite{kaltenbacher2008iterative}, \cite{schuster}. Here, $F'(\pmb{\beta})^*$ denotes the adjoint operator.

We derive the procedure to compute the gradient \eqref{gradientOfObjectiveFunctional} for the underlying problem, utilize it for the PQN method and conclude the article with numerical results.

\begin{subsection}{IHCPs}

The determination of heat fluxes from internal enthalpy measurements belongs to the class of Inverse Heat Transfer Problems (IHTPs), which by itself is studied extensively in the current literature, especially in industrial applications. More specifically, these problems are classified as Inverse Heat Convection Problems, which we abbreviate by IHCPs in this paper, because the Neumann boundary conditions are convective type. We note, that this should not be confused with Inverse Heat Conduction Problems, which consist of determining the heat conduction behavior inside of the material by, for example, the identification of the material parameters $k$ and/or $C$ from additional data. We refer the interested reader to standard textbooks as \cite{alifanov} or \cite{ozisik2000inverse} regarding IHTPs in general.

Of course, there are publications on IHCPs in order to determine heat fluxes in very specific settings, involving different measurement techniques for various models in 1D, 2D or 3D, including mathematical questions concerning the conditions on the existence and uniqueness of a solution, etc. However, the heat fluxes (or variations of it) many times dependent only on the time variable $t$, or only the space variable $x$ or both, cf. \cite{su2004inverse}, \cite{huang1992inverse}, \cite{yang2011nonlinear}. Models with temperature- (or enthalpy-) dependent functions are more commonly studied in nonlinear heat \textit{conduction} problems, where e.g. the unknown material parameters such as density, specific heat capacity or thermal conductivity are temperature dependent due to phase changes in the material and have to be determined from a measured temperature state. Often the parametrization approaches have the idea to simplify the functions in a way that will only work in experimental simulations. In \cite{mierzwiczak2011determination} or \cite{cuiGaoZhang} for instance, the authors analyze the parametrized model with low order polynomial functions of temperature in order to determine the unknown coefficients. 
This way, in order to represent more complex functions one needs to increase the order of the polynomial significantly, which has its own drawbacks as numerical instabilities arise due to fact that the coefficients of the high order monomials might be very small. In \cite{rothermel2019parameter}, the authors propose a parameter estimation approach to identify continuously differentiable material parameters $k(\theta)$ and $C(\theta)$ (up to some constant) from internal temperature measurements without any a-priori information. 

The determination of enthalpy- (or temperature-) dependent heat fluxes with regard to IHCPs seems not to be studied a lot in the current literature, although the Leidenfrost phenomenon is a well-known problem and gets more attention when it comes to experimental research for various cooling specifications, cf. \cite{gottfried1966leidenfrost}, \cite{bernardin2004leidenfrost}.
\end{subsection}

\begin{subsection}{Outlook}
In Section \ref{sec:ExistenceAndUniqueness}, we discuss the existence and uniqueness of solutions $u$ of the IBVP \eqref{IBVP1} - \eqref{IBVP4} by pursuing different approaches. On the one hand, we investigate classical solutions in 1D settings by applying the theory proposed by Ladyzhenskaya et al., see \cite{ladyzenskaja1968linear}. On the other hand, we scrutinize the existence of so-called strong solutions by transforming the IBVP in an abstract Cauchy problem and utilizing the approach of pseudomonotone operators introduced by Brezis in \cite{brezis1968equations}. This leads to a solution in a weaker sense, i.e. the requirements on the involved functions in \eqref{IBVP1} - \eqref{IBVP4} are not as strict. We finish Section \ref{sec:ExistenceAndUniqueness} by specifying the appropriate solution space $\mathcal{U}$.

In Section \ref{sec:inverseHCP} we explain, how we represent $\beta_0$ and $\beta_L$ by some parameter vector $\pmb{\beta}\in \mathbb{R}^{2n}$ for $n>0$, in order to define the \textit{parameter-to-solution operator} 
\begin{align}
S:B\subset\mathbb{R}^{2n}&\to \mathcal{U},\nonumber\\
\pmb{\beta}&\mapsto u, \label{parameterToSolution}
\end{align} 
mapping the heat flux parameter to a solution of the associated IBVP \eqref{IBVP1}-\eqref{IBVP4}. We also define the so-called \textit{observation operator}, which encodes the measuring process and makes solutions comparable to the measurement data $u^{\delta}$. Furthermore, we discuss the inverse problem (IHCP) consisting of determining the heat flux parameter $\beta\in \mathbb{R}^{2n}$.

The aim of Section \ref{sec:implementationAndNumericalResults} is to implement the iterative Projected Quasi-Newton method proposed by Kim et al., see \cite{kim2010tackling}. We finish the article with numerical results.

\end{subsection}

\section{Setting the stage}\label{sec:ExistenceAndUniqueness}

\subsection{Existence and uniqueness of a classical solution}
For this subsection, let $I=(0,T)$ be the open time interval. Furthermore let $\bar{Q}:=\bar{I}\times \bar{\Omega}$ and $l\in (0,1)$. With $H^{2+l,1+l/2}(\bar{Q})$ we denote the H\"older space consisting of continuous functions $u(t,x)$ in $\bar{Q}$, such that all derivatives with respect to space (resp. time) are continuous up to order 2 (resp. 1). Thus, $u\in \mathcal{C}^{2,1}(\bar{Q})$, but additionally, these functions are also H\"older continuous in space (resp. time) with exponent $l$ (resp. $l/2$), i.e. $\exists \ C>0$ with
\begin{align}
    \label{spaceTimeHoelder}
    \max\limits_{x_1,x_2 \in \Omega} |u(t,x_1)-u(t,x_2)|&\leq C |x_1-x_2|^l, \ \ \forall t\in \bar{I}, \\
        \max\limits_{t_1,t_2 \in I} |u(t_1,x)-u(t_2,x)|&\leq C |t_1-t_2|^{l/2}, \ \ \forall x\in \bar{\Omega}.
\end{align}

In order to show the existence and uniqueness of a classical solution $u\in H^{2+l,1+l/2}(\bar{Q})\subset \mathcal{C}^{2,1}(\bar{Q})$ we modify the original result by Ladyzhenskaya et al. (\cite{ladyzenskaja1968linear}) to fit our situation in the following lemma.

\begin{lemma}\label{lemma:lady}
Given the second order initial boundary value problem 
\begin{align}
    \label{lady1}
     u_t &= a(x,t,u)u_{xx}-b(x,t,u,u_x),  &t \in I,\  x\in \Omega,\\
     \label{lady2}
    a(0,t,u)u_x &= \Phi(0,t,u), &t \in I,\  x=0,\\
    \label{lady3}
    -a(L,t,u)u_x &= \Phi(L,t,u), &t \in I,\  x=L,\\
    \label{lady4}
    u&=\Phi_0, &t=0, x\in \bar{\Omega},
\end{align}
let the following conditions hold:
\begin{align} \label{zwei37}
        a(x,t,u)&> 0  &\text{for } (t,x)\in \bar{\Omega}\times (0,T]\\
    -u b(x,t,u,0) &\leq c_1 u^2 + c_2 &\text{for }(t,x) \in \Omega \times (0,T], \ c_1, c_2\geq0\\
    -u\Phi(x,t,u)&<0 &\text{for } |u|>0,\  (t,x) \in \bar{I}\times \{0,L\}
    \end{align}
Moreover, for arbitrary $p$, $(t,x)\in \bar{Q}$ and $|u|< \infty$, let the functions $a(x,t,u)$, $b(x,t,u,p)$ and $\Phi(x,t,u)$ be continuous in their arguments (and possess all the subsequent derivatives) fulfilling
\begin{align}\label{zwei40}
    a(x,t,u) &\leq c,\\\label{zwei41}
    |a_u,a_x,a_t,a_{uu},a_{ut},a_{ux},a_{xt}|&\leq c,\\\label{zwei42}
    |\Phi,\Phi_u,\Phi_x,\Phi_t,\Phi_{uu},\Phi_{ut},\Phi_{ux}|&\leq c,\\\label{zwei43}
    |b(x,t,u,p)|&\leq c(1+p^2),\\\label{zwei44}
    |b_p|(1+|p|)+|b_u|+|b_t|&\leq c(1+p^2) &\text{ for } c>0.
\end{align}
Furthermore, for $|p|<\infty$ and $l\in(0,1)$, let $a_x(x,t,u)$ and $b(x,t,u,p)$ be H\"older continuous in $x$ with exponent $l$, and $\Phi_x(x,t,u)$ H\"older continuous in $x$ and $t$ with exponents $l$ and $l/2$, respectively. Finally, if 
\begin{align}\label{zwei45}
    \Phi(0,0,0)=\Phi(L,0,0)=0
\end{align}
holds, the initial boundary value problem \eqref{lady1} - \eqref{lady4} has a unique solution $u\in H^{2+l,1+l/2}(\bar{Q})$.
\begin{proof}
Cf. Theorem 7.4 in \cite{ladyzenskaja1968linear}.
\end{proof}
\end{lemma}

\begin{remark}
Note, that in order to better distinguish the derivatives in the above lemma, we used the subscript $u$ to indicate the derivative with respect to $u$. However, we also use the prime notation to refer to this derivative if its clear which derivative is meant, e.g., for $\alpha(u)$ we write $\alpha'(u)$ instead of $\alpha_u(u).$
\end{remark}

\begin{theorem}\label{theoremClassicalSolution}
Let $\alpha \in \mathcal{C}^3$ with
\begin{align}
    0<c_1\leq\alpha'\leq c_2<\infty \label{zwei46}
\end{align}
and \begin{align}
    |\alpha''|\leq c_3,\ \ |\alpha'''|\leq c_4, \label{zwei47}
\end{align}
and $0<\beta_{0},\beta_{L}\in \mathcal{C}^2$ with
\begin{align}\label{zwei48}
    |\beta_0,\beta_L|\leq c_5,\ \ \ |\beta_0',\beta_L'|\leq c_6,\ \ \ |\beta_0'',\beta_L''|\leq c_7,\ \ \
\end{align}
hold for some constants $c_1,\dots,c_7>0.$
If further
\begin{align}
    \beta_0(0)=\beta_L(0)=0,\label{zwei49}
\end{align}
then the IBVP \eqref{IBVP1} - \eqref{IBVP4} has a unique solution $u \in H^{2+l,1+l/2}(\bar{Q})$ for any $l\in (0,1)$.

\begin{proof} We first reformulate \eqref{IBVP1} to
\begin{align}
    \label{reformulationIBVP1}
    u_t = \alpha'(u)u_{xx}+\alpha''(u)u_x^2.
\end{align}
By choosing 
\begin{align*}
    a(x,t,u) = \alpha'(u), \ \ \ b(x,t,u,p) = - \alpha''(u)p^2, \ \ \ \Phi_0 = u_0,  \ \ \ \Phi(0,t,u) = \beta_0(u) \text{\ \ \ and \ \ \ } \Phi(L,t,u) = \beta_L(u),
\end{align*}
the conditions \eqref{zwei37} - \eqref{zwei40} and \eqref{zwei45} follow immediately together with the requirements \eqref{zwei46} and \eqref{zwei49}. The boundedness of the derivatives \eqref{zwei41} - \eqref{zwei42} are a consequence of \eqref{zwei47} - \eqref{zwei48}. Furthermore, for $c:=2(2c_3+c_4)$ we have
\begin{align*}
    |b(x,t,u,p)|&=|-\alpha''(u)p^2|\leq c_3 p^2, \ \text{ and} \\
    |b_p|(1+|p|)+|b_u|+|b_t|&=|2\alpha''(u)p|(1+|p|)+|-\alpha'''(u)p^2|\\
    &\leq 2c_3(|p|+p^2)+c_4 p^2 \\
    &\leq (2c_3+c_4) (|p|+p^2) \\
      &\leq (2c_3+c_4) (2+p^2+p^2) \\
        &\leq c (1+p^2),
\end{align*}
which shows \eqref{zwei43} - \eqref{zwei44}. Since in our case, $a_x=\Phi_x=0$ and $b$ does not depend explicitly on $x$ or $t$, these functions are H\"older continuous with any exponent $l$ or $l/2$ for $l\in (0,1)$, respectively. Thus, the existence of a unique solution $u\in H^{2+l,1+l/2}(\bar{Q})$ follows from Lemma \ref{lemma:lady}.
\end{proof}

\end{theorem}

Obviously, the requirements of Theorem \ref{theoremClassicalSolution} on the involved functions are very restrictive. Furthermore, using classical solutions the treatment of the inverse problem demand for regularization methods in H\"older spaces and hence non-reflexive Banach spaces what would be very intricate (cf. \cite{schuster}). That is why we want to scrutinize the theory of strong solutions. Such solutions exist under milder assumptions and, in our case, even in a Hilbert space setting, which simplifies the further analysis.

For this purpose, following \cite{zeidler2013nonlinear} and \cite{roubicek}, we summarize important mathematical preliminaries in the next subsection.

\subsection{Strong solutions of abstract Cauchy problems}
Let V be a separable and reflexive Banach space with norm  $\|\cdot\|_V$ and
$V^*$ the dual space with norm $\|\cdot\|_{V^*}$. The bilinear form $\left<\cdot,\cdot \right >_{V \times V^*}$ denotes the \textit{dual pairing}.
For $p\geq 1$,\  $L^p(I;V)$ denotes the Bochner space of functions on the bounded time interval $I$ with values in $V$. With the norm
\begin{align}\label{bochnernorm}
\|u\|_{L^p(I;V)}:=\left(\int_0^T \|u(t)\|_V^p \mathrm{d} t\right)^{1/p}
\end{align}
$L^p(I;V)$ turns into a Banach space itself. Furthermore, by the conjugate exponent $p'=p/(p-1)$ the associated dual space is given by $L^{p'}(I;V^*)$. The dual pairing is defined as
\begin{align}
    \label{dualPairingBochner}
    \left<u,v \right >_{L^p(I;V) \times L^{p'}(I;V^*)}:=\int_0^T \left<u(t),v(t) \right >_{V \times V^*} \mathrm{d}t.
\end{align}

Occasionally, if it is evident which norm or dual pairing is concerned we drop the subscripts and simply write
$\|\cdot\|$ or $\left<\cdot,\cdot \right >$. 
As usual, $u_k\to u$ or $u_k\rightharpoonup u$ denotes the \textit{strong} or \textit{weak convergence}.

The theory of so-called \emph{strong solutions} is closely linked to abstract operator equations with pseudomonotone operators. 
These were originally introduced by Brezis in \cite{brezis1968equations} as an extension to the \textit{monotonicity} theory of Minty-Browder (\cite{minty1963monotonicity}), which guarantees the existence of solutions $u\in X$ for abstract operator equations
\begin{align}
    \label{operatorEq}
    A(u)=f \ \ \ \ \text{in } X^*
\end{align}
for monotone operators $A:X\to X^*$, where $X$ is a reflexive Banach space and $X^*$ its dual. Following \cite{zeidler2013nonlinear} and \cite{roubicek}, the results of Brezis can also be extended to evolutionary problems of the form
\begin{align}
\label{cauchyInVStern}
    u_t+A(u(t))&=f(t) \ \ \ \text{ for a.a. } t\in I,\\
    u(0)&=u_0.\label{cauchyInH}
\end{align}

Considering the \textit{abstract Cauchy problem} \eqref{cauchyInVStern} - \eqref{cauchyInH}, we emphasize that for an operator $A:V\to V^*$ the first equation is to be understood as an equation in $V^*$, while the second equation holds in some separable Hilbert space H, where $V\subset H \subset V^*$ forms a \textit{Gelfand triple}, i.e. both the embeddings $V\hookrightarrow  H$ and $H\subset V^*$ are continuous and dense, see \cite{wlokapartial}. This way, the dual pairing acting on $V^*\times V$ is a continuous extension of the inner product $(\cdot,\cdot)$ on the Hilbert space H, i.e. we have 
\begin{align} \label{innerproductExtension}
    (u,v)=\langle u, v\rangle \ \ \ \ \ \text{ for } u \in H \text { and } v \in V. 
\end{align}

Inspecting equation \eqref{cauchyInVStern} further, we realize that, if $u\in L^{p}(I;V)$, it makes sense to assume that $f$ and $u_t\in L^{p'}(I;V^*)$. For this reason, it is convenient to define the so-called Sobolev-Bochner space
\begin{align}
    \label{sobolevBochnerSpace}
    W^{1,p,p'}(I;V,V^*):=\left\{u\in L^p(I;V);\  u_t \in L^{p'}(I;V^*)\right\},
\end{align}
identifying $u_t$ as the distributional derivative of $u$, i.e.
\begin{align}
    \label{distributionalDerivative}
    u_t(\phi)=-\int_0^T u(t)\phi'(t) \mathrm{d}t,\ \ \  \forall \phi \in \mathcal{C}_0^{\infty}([0,T]).
\end{align}

It is worth mentioning, that this way the embedding $W^{1,p,p'}(I;V,V^*)\hookrightarrow  \mathcal{C}(I;H)$ is continuous.\footnote{The case $p=2=p'$ represents the embedding result of the famous Lions-Magenes lemma, cf. \cite{lions2012non}. A proof for $p\neq 2$ can be found in \cite{roubicek}.} Hence, the evaluation at time $t=0$ makes sense, i.e. \eqref{cauchyInH} is justified.
\begin{dfn}
We call $u \in W^{1,p,p'}(I;V,V^*)$ a \textit{strong solution} to the abstract Cauchy problem \eqref{cauchyInVStern} - \eqref{cauchyInH}, if the first equation holds in $V^*$ and the second equation in $H$.
\end{dfn}

\begin{remark}
Note, that in this paper we sometimes refer to $u \in V$ or $u\in L^p(I;V)$ or $u \in W^{1,p,p'}(I;V,V^*)$ or even $u \in \mathbb{R}_+$. This is however not unusual, mentioning a quote from \cite{zeidler2013nonlinear}:
'\textit{Use several function spaces for the same problem. The modern strategy for nonlinear pde's.}'
For example, depending on the situation, the enthalpy variable $u$ might represent $u(t):=u(t,\cdot) \in V$, $u \in L^p(I;V)$ or simply $u(t,x)\in \mathbb{R}_+$. For brevity, we often omit the variables $t$ and $(t,x)$ and leave it to the reader to put the equations in the right context. 
\end{remark}

For the remainder of this subsection, we aim to facilitate an existence and uniqueness result from \cite{roubicek} concerning strong solutions of abstract Cauchy problems of the form \eqref{cauchyInVStern}-\eqref{cauchyInH}.

\begin{dfn} An operator $A:V\to V^*$ is called
\begin{itemize}
\item[(i)] \textit{strongly continuous}, iff
\begin{align}\label{stronglyCont}
u_k \rightharpoonup u \text{ in } V \quad \Rightarrow \quad A(u_k) \to A(u) \text{ in } V^*, \text{ and}
\end{align}
    \item[(ii)] \textit{pseudomonotone}, iff $A$ is bounded and if the condition
\begin{align*}
\left\{
\begin{array}{r r}
u_k &\rightharpoonup u\\
 \limsup\limits_{k\rightarrow \infty} \left<A(u_k),u_k-u\right> &\leq 0 
\end{array}
\right\}
\end{align*} 
implies that
\begin{align}\label{pseudoA}
\left<A(u),u-v\right> \leq \liminf\limits_{k\rightarrow \infty} \left<A(u_k),u_k-v\right>, \ \ \ \forall v \in V. 
\end{align}
\end{itemize}
\end{dfn}

\begin{proposition}\label{proposition:strongAndPseudoImplications}
For $A:V\to V^*$ and $B:V\to V^*$ the following implications hold:
\begin{itemize}
\item[(i)] $A$ is strongly continuous $\Rightarrow$  $A$ is pseudomonotone.
\item[(ii)] $A$, $B$ are pseudomonotone $\Rightarrow$  $A+B$ is pseudomonotone.
\end{itemize}
\end{proposition}
\begin{proof}
Cf. Proposition 27.6 in \cite{zeidler2013nonlinear}.
\end{proof}

\begin{dfn}\label{weaklyCoercive}
An operator $A:V\to V^*$ is called \textit{weakly coercive}, if
$$\langle A(u),u\rangle \geq c_1\|u\|_V^p-c_2 \|u\|_H^2$$
holds for $c_1,c_2>0$.
\end{dfn}

 \begin{theorem}[Existence and uniqueness of a strong solution]\label{existence_roubicek} \ \\
 Let $A:V\to V^*$ be a pseudomonotone and weakly coercive operator, that satisfies the growth condition
 \begin{align}\label{growthCondition}
 \|A(u)\|_{V^*}\leq \varphi(\|u\|_H)\left(1+\|u\|_V^{p-1}\right)
 \end{align}
 for all $u \in V$ and some increasing function $\varphi:\mathbb{R} \to \mathbb{R}$.
 Then, if $f \in L^{p'}(I;V^*)$ and $u_0 \in H$, the abstract Cauchy problem \eqref{cauchyInVStern} - \eqref{cauchyInH} has a strong solution $u\in W^{1,p,p'}(I;V,V^*)\subset \mathcal{C}(I;H).$
 
 If additionally,
\begin{align} \label{uniquenessRequirement}
 \exists c > 0\  \forall u,v \in V \text{, for a.a. } t \in I: \langle A(u)-A(v),u-v \rangle \geq -c\|u-v\|^2_H,
 \end{align}
 the strong solution $u$ is unique.
 \end{theorem}
 \begin{proof}
Cf. theorems 8.9 and 8.34 in \cite{roubicek}.
 \end{proof}

\subsection{Existence and uniqueness of a strong solution}

In this subsection we investigate strong solutions with regard to the IBVP \eqref{IBVP1} - \eqref{IBVP4}.

For this purpose, we specify the Gelfand triple, i.e. the spaces
\begin{align}
    \label{VundH}
    V:=W^{1,2}(\Omega), \ \ \ \ H:=L^2(\Omega),
\end{align}
where the Sobolev space $W^{1,2}(\Omega)$ is even a reflexive Hilbert space. For $u\in V$ it follows from $\|u\|_V^2=\|u\|_H^2+\|u_x\|_H^2$ that $u \in H$ and $u_x \in H$. Note also, that for $dim(\Omega)=1$ the compact embedding 
\begin{align}
    \label{sobolevEmbedding}
    V \subset \subset \mathcal{C}(\bar{\Omega})
\end{align}
holds true, cf. e.g. \cite{brezis2010functional}, \cite{wlokapartial}.

We perform a \textit{weak formulation} with respect to the space variable $x\in \Omega$. Multiplying \eqref{IBVP1} by $v\in V$, integrating over $\Omega$ and using integration by parts lead to
\begin{align}
\int_\Omega u_t(t,x)v(x)\ \mathrm{d} x-\int_{\partial \Omega}\left(\alpha'(u(t,x))u_x(t,x)\right)v(x)\ \mathrm{d} S \nonumber \\ +\int_\Omega \left(\alpha'(u(t,x))u_x(t,x)\right)v_x(x)\ \mathrm{d} x=0, \label{greenFirst}
\end{align}
where for example the first integral can also be expressed as $\langle u_t,v\rangle_{V^*\times V}$, cf. \eqref{innerproductExtension}. With $\partial \Omega$ we denote the boundary on $\Omega$, which in our 1D case is just the two points $x=0$ and $x=L.$ Since \eqref{sobolevEmbedding} holds, these point evaluations are well-defined. Inserting the boundary conditions \eqref{IBVP2}-\eqref{IBVP3}, we are able to formulate an abstract Cauchy problem
\begin{align}
    \label{insertBoundaryConditionsCauchyProblem1}
    u_t + A(u(t)) &= 0 \ \ \ \ \text{ in } V^*,\\
    \label{insertBoundaryConditionsCauchyProblem2}
    u(0) &= u_0 \ \ \text{ in } H,
\end{align}
for $A:V\to V^*$ defined by
\begin{align}
    \label{operatorA}
    \langle A(u(t)),v\rangle_{V^*\times V}:= \int_\Omega \left(\alpha'(u(t,x))u_x(t,x)\right)v_x(x)\ \mathrm{d} x+\beta_L(u(t,L))v(L)+\beta_0(u(t,0))v(0),
\end{align}
where the initial condition \eqref{IBVP4} holds in the Hilbert space setting, i.e. $u_0 \in H$.

Since enthalpy is positive and we are extracting heat on the boundaries (in the underlying cooling process setting) we presume that  $u(t,x)\in [0,u_{max}]$ for all  $(t,x)\in I\times \Omega$, where $u_{max}\geq \esssup\limits_{x \in \bar{\Omega}}u_0(x)$. We furthermore assume the following function qualifications:
\begin{align}
\label{betaDifferenzierbar}
    0<\beta_0, \beta_L,\alpha \in \mathcal{C}^1([0,u_{max}]),\\ \label{beta_maxDefinition}
    \max \left\{\max\limits_{u}\beta_0(u),\max\limits_{u}\beta_L(u)\right\}=:\beta_{max}<\infty,\\
    \label{betaStrich_maxDefinition}
    \max \left\{\max\limits_{u}\beta'_0(u),\max\limits_{u}\beta'_L(u)\right\}=:\beta'_{max}<\infty,\\     \label{alphaStrichBeschraenkt}  
    0<c_{min}:=\min\limits_{u} \alpha'(u) \leq \alpha'(u) \leq c_{max}:=\max\limits_{u} \alpha'(u)<\infty.
\end{align}

Note, that because of \eqref{alphaStrichGleichTLF} and since the thermal diffusivity of materials is positive and naturally bounded, the requirement \eqref{alphaStrichBeschraenkt} is physically justified. This way, together with $u_x,v_x\in H$,  the integral in \eqref{operatorA} exists and the operator $A$ is well-defined. We also emphasize that $\beta_{max}>0$ is the a-priori known upper bound that we use later for the box-constraint $\pmb{\beta} \in B:=[0,\beta_{max}]^{2n}$.

We proceed by examining the existence and uniqueness of a strong solution to \eqref{insertBoundaryConditionsCauchyProblem1} - \eqref{insertBoundaryConditionsCauchyProblem2} utilizing Theorem \ref{existence_roubicek}. We first show the operator defined by \eqref{operatorA} is pseudomonotone:
By splitting $A:V\to V^*$ into two operators defined by
\begin{align}
\label{A1Definition}
\left<A_1(u),v\right>_{V^*\times V}:=\int_\Omega \alpha'(u)u_x v_x\  \mathrm{d}x \ \ \text{ and}
\end{align}
\begin{align}
    \label{A2Definition}
    \left<A_2(u),v\right>_{V^*\times V}:=\beta_L(u(\cdot,L))v(L)+\beta_0(u(\cdot,0))v(0)
\end{align}
such that $A=A_1+A_2$, it suffices to show the pseudomonotonicity of $A_1$ and $A_2$, cf. Proposition \ref{proposition:strongAndPseudoImplications}, (ii). Since for $u_k\rightharpoonup u$ we have $\beta_0(u_k)\to \beta_0(u)$ and $\beta_L(u_k)\to \beta_L(u)$, it follows that $\langle A_2(u_k),v\rangle \to \langle A_2(u),v \rangle$. Thus, the operator $A_2$ is strongly continuous and consequently pseudomonotone, cf. \eqref{stronglyCont} and Prop. \ref{proposition:strongAndPseudoImplications}, (i). To show \eqref{pseudoA} for $A_1$ we need the following lemma, which is inspired by Lemma 2.32 in \cite{roubicek}.

\begin{lemma} \label{lemmaForB}
Let $B:V\times V \to V^*$ be given by
\begin{align} 
\langle B(w,u),v\rangle:=\int_\Omega \alpha'(w)u_x v_x \ \mathrm{d}x,
\end{align}
then the following holds:
\begin{itemize}
\item[(i)] $B(u,u)=A_1(u)$ in $V^*$  $\forall u\in V$.
\item[(ii)] For fixed $w\in V$ the operator $B(w,\cdot):V\to V^*$ is monotone, i.e.
$$\langle B(w,u)-B(w,v),u-v\rangle \geq 0 \ \ \ \forall u,v \in V.$$
 
\end{itemize}
Furthermore, for $u_k \rightharpoonup u \in V$ it holds
\begin{itemize} 
\item[(iii)]  $\lim\limits_{k\to \infty} \langle B(u_k,v),w \rangle =\langle B(u,v),w \rangle$,\   $\forall v,w\in V$.
\item[(iv)] $\lim\limits_{k\to \infty} \langle B(u_k,v),u_k-u \rangle =0$,\   $\forall v\in V$.
\end{itemize}
\begin{proof}
The identity in (i) is obvious and monotonicity in (ii) follows from
\begin{align*}
\langle B(w,u)-B(w,v),u-v\rangle =& \int_\Omega \alpha'(w)(u-v)_x(u-v)_x\  \mathrm{d} x\\ \overset{\eqref{alphaStrichBeschraenkt}}{\geq} & c_{min} \left\|(u-v)_x\right\|_H^2 \geq 0.
\end{align*}
For the remaining proof we refer to Lemma 2.32 in \cite{roubicek}.
\end{proof}
\end{lemma}

\begin{lemma}\label{A1isPseudomonotone}
The operator $A_1:V \to V^*$ defined by \eqref{A1Definition} is pseudomonotone.
\begin{proof}
For $u\in V$ the operator is bounded, since 
\begin{align}
    \|A_1(u)\|_{V^*}&=\sup\limits_{v \in V,\ \|v\|\leq 1} \left|\left<A_1(u),v\right>\right|\nonumber\\& \overset{\eqref{alphaStrichBeschraenkt}}{\leq} \sup\limits_{v \in V,\ \|v\|\leq 1} \left(c_{max} \left\|u_x\right\|_H \left\|v_x\right\|_H\right)\nonumber\\
&\leq \sup\limits_{v \in V,\ \|v\|\leq 1} \left(c_{max} \left\|u\right\|_V \left\|v\right\|_V\right)\nonumber\\
&\leq c_{max}\|u\|_V. \label{alreadyShownForA1}
\end{align}
Now we show \eqref{pseudoA} for $A_1$ by assuming that \begin{align}
\left\{
\begin{array}{r r}
u_k &\rightharpoonup u\\
 \limsup\limits_{k\rightarrow \infty} \left<A_1(u_k),u_k-u\right> &\leq 0 
\end{array}
\right\}. \label{pseudoRequirement}
\end{align}
To be thorough, we execute the proof similar as in the abovementioned reference. We first define
\begin{align}
\label{uepsilon}
u_\varepsilon:=(1-\varepsilon)u+\varepsilon v \ \ \ \in V
\end{align}
for $\varepsilon \in (0,1)$. Using the monotonicity of the operator B for $w=u_k$, $u=u_k$ and $v=u_{\varepsilon}$, we obtain
\begin{align} \label{4zweiundzwanzig}
\langle A_1(u_k),u_k-u_{\varepsilon}\rangle\ \geq & \ \langle B(u_k,u_{\varepsilon}),u_k-u_{\varepsilon}\rangle,
\end{align}
since $B(u_k,u_k)=A_1(u_k)$, cf. Lemma \ref{lemmaForB}, (i) and (ii). Inserting \eqref{uepsilon} leads to
\begin{align*}
\varepsilon \langle A_1(u_k),u-v \rangle \geq & -\langle A_1(u_k),u_k-u \rangle \\ +& \ \langle B(u_k,u_{\varepsilon}),u_k-u\rangle + \varepsilon \langle B(u_k,u_{\varepsilon}),u-v\rangle.
\end{align*}
Considering the limit with respect to $k\to\infty$ on both sides of the inequality and using the second line of \eqref{pseudoRequirement}, together with (iii) and (iv) of Lemma \ref{lemmaForB} we get
\begin{align}\label{durchEpsilonZuTeilen}
\varepsilon \liminf\limits_{k\to \infty}\langle A_1(u_k),u-v \rangle \geq \varepsilon \langle B(u,u_{\varepsilon}),u-v\rangle .
\end{align}
Dividing by $\varepsilon>0$ and passing the limit $\varepsilon\to 0$ lead to $u_{\varepsilon}\to u$ and therefore to
\begin{align}
\label{fastDa}
\liminf\limits_{k\to \infty}\langle A_1(u_k),u-v \rangle \geq \langle A_1(u),u-v \rangle.
\end{align} 
We use the monotonicity of $B$ again, this time with $v=u$, to get
\begin{align} \label{45undzwanzig}
\langle B(u_k,u_k)-B(u_k,u),u_k-u\rangle \geq 0.
\end{align}
With \eqref{fastDa}, \eqref{45undzwanzig} and (iv) from Lemma \ref{lemmaForB} we can finally show \eqref{pseudoA} by
\begin{align*}
\liminf\limits_{k\to \infty}\langle A_1(u_k),u_k-v \rangle = \liminf\limits_{k\to \infty}\langle A_1(u_k),u_k-u \rangle +\liminf\limits_{k\to \infty}\langle A_1(u_k),u-v\rangle \\ = \lim\limits_{k\to \infty}\langle B(u_k,u),u_k-u \rangle +\liminf\limits_{k\to \infty}\langle B(u_k,u_k)-B(u_k,u),u_k-u \rangle\\+ \liminf\limits_{k\to \infty}\langle A_1(u_k),u-v \rangle \geq \langle A_1(u),u-v \rangle.
\end{align*}
\end{proof}
\end{lemma}

\begin{corollary}
The operator $A=A_1+A_2$ is pseudomonotone as a sum of two pseudomonotone operators.
\end{corollary}
\ \\ \vspace{-9mm}

\begin{lemma}\label{lemma:coercive}
If $p=2$, then $A:V\to V^*$ from \eqref{operatorA} is weakly coercive and satisfies the growth condition \eqref{growthCondition}.
\end{lemma}
\begin{proof}
The weak coercivity follows immediately from $\beta_0,\beta_L,u>0$ and \eqref{alphaStrichBeschraenkt}: 
\begin{align*}
    \langle A(u),u \rangle \geq c_{min} \|u_x\|_H^2=c_{min} \|u\|_V^2-c_{min} \|u\|_H^2.
\end{align*}
For $u \in V$ we can further estimate
\begin{align*}
\|A(u)\|_{V^*}=&\sup\limits_{v \in V,\ \|v\|\leq 1} \left|\left<A(u),v\right>\right|\\ \underset{\eqref{beta_maxDefinition}}{\overset{\eqref{alreadyShownForA1}}{\leq}} &c_{max}\|u\|_V+\sup\limits_{v \in V,\ \|v\|\leq 1} \left(2\beta_{max} \sup\limits_{x \in \bar{\Omega}}|v(x)|\right)\\
\leq &c_{max}\|u\|_V+2\beta_{max}\sup\limits_{v \in V,\ \|v\|\leq 1} \left( c_{c}\|v\|_V\right)\\
\leq &c\left(1+\|u\|_V \right),
\end{align*}
where $c:=\max\left\{c_{max},2\beta_{\max}c_{c}\right\}$ and $c_{c}$ is the norm of the embedding $V\hookrightarrow  \mathcal{C}(\bar{\Omega})$, c.f. \eqref{sobolevEmbedding}. This shows \eqref{growthCondition} for $p=2$ since $c>0$ can be replaced by any appropriate function $\varphi$.
\end{proof}
Note that we did not specify the functional setting with respect to the time variable so far. Lemma \ref{lemma:coercive} suggests the exponent $p=2$ and defining the underlying Bochner space as $L^2(I;V)$. Since $p'=2$, we can also fix the Sobolev-Bochner space
\begin{align}
    \mathcal{U}:= W^{1,2,2}(I;V,V^*),
\end{align}
which is even a Hilbert space (see \cite{alphonse2014abstract}),
as the space of possible strong solutions to \eqref{insertBoundaryConditionsCauchyProblem1} - \eqref{insertBoundaryConditionsCauchyProblem2}.

\begin{theorem}\label{mainTheorem_Existence_Uniqueness}
The abstract Cauchy problem \eqref{insertBoundaryConditionsCauchyProblem1} - \eqref{insertBoundaryConditionsCauchyProblem2} has a strong solution $u\in \mathcal{U}\subset \mathcal{C}(I;H).$
If additionally there exists a constant 
\begin{align}
    \label{alphaStrichConst}
     \alpha'_c\geq c_u:=2\beta'_{max}c_c^2>0
\end{align}
with $\beta'_{max}$ from \eqref{betaStrich_maxDefinition} and where $c_c$ is the norm of the embedding $V\hookrightarrow  \mathcal{C}(\bar{\Omega}),$ such that
\begin{align}
    \label{requirementForA1}
    \int_{\Omega} \left(\alpha'(u)u_x-\alpha'(v)v_x\right)(u_x-v_x)\ \mathrm{d}x \geq \alpha'_c \int_{\Omega} (u_x-v_x)^2\ \mathrm{d}x = \alpha'_c\|u_x-v_x\|_H^2, \ \ \  \forall u,v \in V,
\end{align}
then the strong solution is even unique.
\begin{proof}
We aim to apply Theorem \ref{existence_roubicek}. Note that in our case $f\equiv 0$. The existence part of the proof then follows directly from the preliminaries, especially by using Lemma \ref{A1isPseudomonotone} and Lemma \ref{lemma:coercive}.

Furthermore, we can estimate
\begin{align*}
 \langle A(u)-A(v),u-v \rangle &\geq \alpha'_c\|u_x-v_x\|_H^2-|\beta_L(u(L))-\beta_L(v(L))|\cdot|u(L)-v(L)|\\
 &\ \ \ \ \ \ -|\beta_0(u(0))-\beta_0(v(0))|\cdot|u(0)-v(0)|\\
 &\overset{\eqref{betaStrich_maxDefinition}}{\geq} \alpha'_c\|u_x-v_x\|_H^2-2\beta'_{max} c_c^2 \|u-v\|_V^2\\
  &\geq -c_u\|u-v\|^2_H+(\alpha'_c-c_u)\|u_x-v_x\|_H^2\\
  &\overset{\eqref{alphaStrichConst}}{\geq} -c_u\|u-v\|^2_H,
 \end{align*}
which proves the uniqueness condition \eqref{uniquenessRequirement} of the strong solution.
\end{proof}
\end{theorem}

\begin{remark}
Compared to the rigorous requirements on the involved functions to guarantee a classical solution $u\in H^{2+l,1+l/2}(\bar{Q})$, the assumptions \eqref{betaDifferenzierbar} - \eqref{alphaStrichBeschraenkt} for the existence of a strong solution $u\in \mathcal{U}$ are rather weak. We furthermore get the following useful result:

Since $\mathcal{U}\subset L^2(I;V)\subset L^2(I;H)$ we define the Gelfand triple
\begin{align}
    \label{newGelfandTriple}
    \mathcal{U}\subset L^2(I;H)\cong L^2(Q) \subset \mathcal{U}^*.
\end{align}
Similar to \eqref{innerproductExtension}, we identify the duality pairing on $\mathcal{U}^*\times \mathcal{U}$ with the inner product on $L^2(Q)$, i.e.
\begin{align}
    \label{innerproductExtensionOnU}
  \langle u, v\rangle_{\mathcal{U}^* \times \mathcal{U}} = (u,v)_{L^2(Q)} = \int_I\int_{\Omega} u(t,x)v(t,x)\ \mathrm{d}x\ \mathrm{d}t  \ \ \ \ \ \text{ for } u \in L^2(Q) \text { and } v \in \mathcal{U}. 
\end{align}
If we replace $\mathcal{U}$ by the non-reflexive H\"older space $H^{2+l,1+l/2}(\bar{Q})$, then the definition of a Gelfand triple is impossible. The strong solution theory yields \eqref{innerproductExtensionOnU} as a convenient tool in order to compute the gradient of the objective functional, see Section \ref{sec:inverseHCP} for more details. 

Unfortunately, there is a drawback concerning the uniqueness of the strong solution $u\in \mathcal{U}$.
Proving the existence of $\alpha'_c>0$ fulfilling \eqref{alphaStrichConst} seems only being simple for the constant case, i.e. for $\alpha'(u)=\alpha'(v),\  \forall u,v \in \mathbb{R_+}$.

Nevertheless, in order to ensure the uniqueness of a strong solution in the non-constant case, we can still rely on the theory of classical solutions. Since every classical solution is also a solution in some weaker sense, we gain a unique strong solution this way by only requiring stronger assumptions to the coefficients according to Theorem \ref{theoremClassicalSolution}. 
\end{remark}

\section{The inverse heat convection problem}\label{sec:inverseHCP}

In this section we consider the inverse heat convection problem (IHCP) which consists of computing the heat fluxes $\beta_0$ and $\beta_L$ from the knowledge of enthalpy measurements $u^{\delta}$. To solve this problem numerically, we aim to minimize the objective functional $f$ from \eqref{objFunctional} iteratively. To this end we need to derive the gradient \eqref{gradientOfObjectiveFunctional} after specifying the composition \eqref{compositionForF}, i.e. the operators $S$ and $\mathcal{Q}$, in more detail.

\subsection{The parameter-to-solution operator $S$}
\label{subsec:parametrizationApproach}
Given the IBVP \eqref{IBVP1} - \eqref{IBVP4}, the assumptions in the previous section guarantee a strong solution $u\in \mathcal{U}$. As already discussed in the introduction, a solution operator mapping enthalpy-dependent heat fluxes $\beta_0(u), \beta_L(u)\in \mathcal{C}^1([0,u_{max}])$ in \eqref{IBVP2} - \eqref{IBVP3} to the solution $u\in \mathcal{U}$ can only be defined implicitly by
\begin{align}
    \label{implicitlyGivenPTSOperator}
    S(\beta_0(u),\beta_L(u),u)=0
\end{align}
with an appropriate mapping $S$. A sensitivity analysis of the operator $S$ with respect to the heat fluxes could be performed, e.g., by means of the implicit mapping theorem in Banach spaces, cf. \cite{lang2012fundamentals}. We note, that the treatment of implicitly given solution operators and the corresponding inverse problems are a mathematically very interesting task, which seems not to be widely studied in current literature. In this paper, we propose a simple but advantageous approach to numerically treat this problem that allows us to overcome the implicitness and define a parameter-to-solution operator explicitly by
\begin{align}
\label{explicitlyGivenPTSOperator}
    S:B\subset\mathbb{R}^{2n}&\to \mathcal{U},\\
    \pmb{\beta}&\mapsto u.
\end{align}

Heat fluxes $\beta_0,\beta_L$ fulfilling \eqref{betaDifferenzierbar} - \eqref{betaStrich_maxDefinition} are then represented as follows: Let $n>0$ and
\begin{align}
    \label{definitionPIn}
   \pi_n:0=u_1<u_2<\cdots<u_n=u_{max}
\end{align} be a partition of the interval $U:=[0,u_{max}]$. Given $\pi_n$ and a set of values $\pmb{\beta}\in\mathbb{R}^{2n}$, we construct piecewise cubic Hermite interpolating polynomials $\tilde{\beta_0}$ and $\tilde{\beta_L}$ in $\mathcal{C}^1(U)$, so-called PCHIPs, such that 
\begin{align}
\label{interpolationPairs}
        \tilde{\beta_0}(u_i)=\beta_i,  \ \ \ \ 
    \tilde{\beta_L}(u_i)=\beta_{n+i},  \ \ \ \ \text{for\ \ \ \ } i=1,\dots,n.
\end{align}
Choosing $\pmb{\beta}\in \mathbb{R}^{2n}$ appropriately, the approximation properties of piecewise cubic interpolation methods guarantee that for all $\epsilon>0$ there exists a number $n$ of partition points with
\begin{align*}
    \max\limits_{u \in U}|\tilde{\beta_0}(u)-\beta_0(u)|<\epsilon,\\
    \max\limits_{u \in U}|\tilde{\beta_L}(u)-\beta_L(u)|<\epsilon.
\end{align*}
We refer to the Section \ref{pchipsApproximationseigenschaft} for a proof of this approximation property and a numerical example on how to determine the number $n$ of partition points for given $\varepsilon>0$.

Hence, by neglecting a small error, we can replace $\beta_0$ and $\beta_L$ by $\tilde{\beta_0}$ and $\tilde{\beta_L}$. For better readability, we will drop the tilde in $\tilde{\beta_0}$ and $\tilde{\beta_L}$ and just refer to $\beta_0$ and $\beta_L$, noting that we actually mean the interpolants $\beta_0(u,\pmb{\beta})$ and $\beta_L(u,\pmb{\beta}).$ 

\begin{remark}
Compared to a dictionary approach
\begin{align}
    \label{dictionaryRepresentation}
    \beta_0(u)  = \sum\limits_{i=1}^n \beta_i B_{0,i}(u), \ \ \ \ 
     \beta_L(u)  = \sum\limits_{i=1}^n \beta_{n+i} B_{L,i}(u),
\end{align}
where an approximation is made by a linear combination of a-priori known physically meaningful functions
\begin{align}
    \label{dictionaryOfFunctions}
    \left\{B_{0,i}\right\}_{i=1,\dots,n},  \ \ \ \ \left\{B_{L,i}\right\}_{i=1,\dots,n},
\end{align}
the proposed PCHIP approach has a local adjustment property as it carries not weights but rather locally fixed function values. A desired change of the heat fluxes on $[u_a,u_b]\subset U$ only affects parameter components in this subinterval, i.e. only for $i=1,\dots,2n$ with $\beta_i \in [u_a,u_b].$ This might have an additional stabilizing effect on the solution of the dynamic (i.e. time-dependent) inverse problem, especially when enthalpy (or temperature) values are not re-occurring often, as the enthalpy subinterval $[u_a,u_b]$ can approximately be assigned to a time subinterval $[t_a,t_b]\in I$. The optimized heat fluxes can then be trusted in this time interval and the optimization in the next time interval does not affect the previous one, hence achieving some kind of time-marching-scheme effect without really locking the parameters in first place. 
With the PCHIP approach, we also do not limit the scope of the inverse problem by prescribing the structure of the solution too much. Also, increasing the complexity of the solution space for the PCHIP approach by adjusting the number $n>0$ of partition points, rather than the polynomial order in the parametrization approach
\begin{align}
    \label{polynomialParametrizationApproach}
    \beta_0(u) = \sum\limits_{i=1}^n \beta_i u^{i-1},
    \beta_L(u) = \sum\limits_{i=1}^n \beta_{n+i} u^{i-1},
\end{align}
is more stable and numerically superior regarding the IHCP.
\end{remark}

The PCHIP interpolation method itself was first proposed in \cite{fritschCarlson} in order to introduce a piecewise cubic interpolation method which respects the monotonicity of the function values. In this sense, the interpolant is shape-preserving. For solving inverse problems this behavior is favorable because maxima and minima of the interpolant are retained by the function value parameter $\pmb{\beta}\in \mathbb{R}^{2n},$ while other interpolation methods might introduce overshoots and undesirable oscillations. Hence, a-priori informations can be better incorporated. Still, the parametrization approach \eqref{definitionPIn} - \eqref{interpolationPairs} would work for any piecewise interpolation method. As the interpolation method choice is actually of secondary importance, we refer the reader to Section \ref{pchipsKonstruktion} for further construction details on PCHIPs.

Finally, we are able to specify \eqref{explicitlyGivenPTSOperator}.
For some fixed partition $\pi_n$ and the PCHIP interpolation method, $u_0 \in H$ and $\alpha(u)\in \mathcal{C}^1(U)$ the parameter-to-solution operator $S$ is given by
\begin{align}
\label{processChain}
    \pmb{\beta} \in B=[0,\beta_{max}]^{2n} \xrightarrow{\text{Interpolation}} \beta_0,\beta_L\in \mathcal{C}^1(U)\xrightarrow{\text{Inserting}} \eqref{IBVP2}-\eqref{IBVP3} \xrightarrow{\text{Solving the IBVP}} u\in \mathcal{U}.
\end{align}

\subsection{The observation operator $\mathcal{Q}$}\label{subsec:observationOperator}

Usually, the measurement process does not acquire the full enthalpy solution $S(\pmb{\beta})=u\in \mathcal{U}$ of the IBVP \eqref{IBVP1} - \eqref{IBVP4}, but rather some enthalpy state $u_s$. This leads to the definition of the so-called observation operator 

\begin{align}
    \label{observationOperator}
    \mathcal{Q}: \mathcal{U} &\to Y,\\
    u & \mapsto u_s,
\end{align}
which maps $u$ to the measured enthalpy $u_s$. Moreover we assume that only a noise-contaminated version $u^{\delta}$ of $u_s$ is given, cf. \eqref{noiseLevel}.

Recalling the motivation of this paper, in the TMCP process, thermocouples are inserted into the steel plate at depths $x_j\in \Omega$ for $j=1,\dots,d$, which all measure temperatures during the ACC process at time instances $t_i\in I$ for $i=1,\dots,m.$ Thus, from the application point of view, the choice $Y=\mathbb{R}^{d\times m}$ makes sense.
For $u\in \mathcal{U}$ the components of the enthalpy state $u_s \in Y$ are then given by
\begin{align}
\label{specifiedObservationOperator}
    (u_s)_{j,i}=(\mathcal{Q}u)_{j,i} =  u(t_i,x_j).
\end{align}
The point evaluations in time and space were justified in the previous section.

We note that the observation operator is linear, i.e. the Fr\'{e}chet derivative is $\mathcal{Q}'=\mathcal{Q}$. 
With $\delta(\cdot)$ being the Dirac delta function, the adjoint operator
\begin{align}
\label{adjointOfObservationOperator}
    \mathcal{Q}^*: Y=Y^*&\to \mathcal{U}^*,\\
    v & \mapsto \sum\limits_{j,i}\delta(\cdot-t_i)\delta(\cdot-x_j)v_{j,i}
\end{align}
is verified by utilizing the Frobenius inner product $\left(A,B\right)_Y=\sum\limits_{j,i} A_{ji}B_{ji}$ for two real-valued matrices on $Y$ together with the inner product extension \eqref{innerproductExtensionOnU}, such that
\begin{align}
    \label{derivationAdjointOfObservationOperator}
    (\mathcal{Q}u,v)_Y = (u,\mathcal{Q}^*v)_{L^2(Q)}.
\end{align}

\subsection{The linearized and adjoint linearized parameter-to-solution operator}\label{subsec:linearizedAndAdjointLinearizedParameterToSolutionOperator}

In this subsection we deduce the missing ingredients in order to compute the gradient \eqref{gradientOfObjectiveFunctional} and especially the adjoint operator $F'(\pmb{\beta})^*$.

Assuming for the moment that $S$ possesses a Fr\'{e}chet derivative $S'$, from \eqref{compositionForF} it follows that 
\begin{align}
    \label{adjointOperatorOfFequalsSPrimeStarQStar}
    F'(\pmb{\beta})^* = S'(\pmb{\beta})^*\circ Q^*,
\end{align}
which is a standard result in functional analysis, cf. \cite{rudin1991functional}.

We start by linearizing the parameter-to-solution operator $S$ yielding
\begin{align}\label{gateaux1}
    \tilde{S}_{\pmb{\beta}}:\mathbb{R}^{2n}&\to \mathcal{U},\\ \label{gateaux2}
    \mathbf{h}&\mapsto \lim\limits_{\epsilon\to 0}\frac{S(\pmb{\beta}+\epsilon \mathbf{h})-S(\pmb{\beta})}{\epsilon},
\end{align}
and by replacing $S'(\pmb{\beta})$ with $ \tilde{S}_{\pmb{\beta}}$ in \eqref{adjointOperatorOfFequalsSPrimeStarQStar}. Presuming that $(\pmb{\beta}+\epsilon\mathbf{h})\in B$ and that the limit in \eqref{gateaux2} exists,  
\begin{align}
    \label{wEqualsLinearizedSbetaH}
w:=\tilde{S}_{\pmb{\beta}}\pmb{h}
\end{align}
represents the G\^{a}teaux derivative, i.e. the directional derivative of $S$, evaluated in $\pmb{\beta}\in B$ with respect to $\mathbf{h}\in \mathbb{R}^{2n}$. For a fixed direction $\mathbf{h}$ we have $\tilde{S}_{\pmb{\beta}}\mathbf{h}=S'(\pmb{\beta})\mathbf{h}$ and consequently
\begin{align}\label{adjungierteSindGleich}
\tilde{S}_{\pmb{\beta}}^*v=S'(\pmb{\beta})^*v,\ \forall v \in \mathcal{U}^*.
\end{align}

In the following we aim to derive the computation of \eqref{wEqualsLinearizedSbetaH} and the adjoint state $\tilde{S}_{\pmb{\beta}}^*v$ for $v \in \mathcal{U}^*$, where
\begin{align} \label{adjointOperatorOfLinearizedParameterToSolutionOperator}
    \tilde{S}_{\pmb{\beta}}^*:\mathcal{U}^*\to \mathbb{R}^{2n}
\end{align} is the adjoint operator which fulfills
\begin{align}
\label{definitionAdjointOperator}
    \langle \tilde{S}_{\pmb{\beta}}\mathbf{h},v\rangle_{\mathcal{U}\times \mathcal{U}^*} = ( \mathbf{h},\tilde{S}_{\pmb{\beta}}^*v)_2, \ \ \forall v \in \mathcal{U}^*.
\end{align}

\begin{lemma}
For some fixed $\pmb{\beta}\in B$, let $u:=S(\pmb{\beta})$ be the associated solution of the IBVP \eqref{IBVP1} - \eqref{IBVP4}. The operator $\tilde{S}_{\pmb{\beta}}$
then maps $\mathbf{h}\in \mathbb{R}^{2n},$ with $(\pmb{\beta}+\epsilon\mathbf{h})\in B$, to the solution $w$ of the initial boundary value problem
\begin{align}
      \label{gateauxPDE1}
    w_t &= (\alpha'(u)w)_{xx},  &t \in I,\  x\in \Omega,\\
    (\alpha'(u)w)_x &= \beta_0'(u,\pmb{\beta})w+\nabla  \beta_0(u,\pmb{\beta})\cdot \mathbf{h}, &t \in I,\  x=0,\label{gateauxPDE2}\\
    -(\alpha'(u)w)_x &= \beta_L'(u,\pmb{\beta})w+\nabla \beta_L(u,\pmb{\beta})\cdot \mathbf{h}, &t \in I,\  x=L,\label{gateauxPDE3} \\
    w&=0, &t=0, x\in \bar{\Omega}. \label{gateauxPDE4}
\end{align}
Here $\beta_0'$ and $\beta_L'$ are the derivatives with respect to the first variable, while $\nabla \beta_0$ and $\nabla \beta_L$ are the gradients with respect to the second variable.
\end{lemma}
\begin{proof}
The initial boundary value problem \eqref{gateauxPDE1} - \eqref{gateauxPDE4} is obtained in the following way. Let $\epsilon>0$. The perturbation of the parameter $\pmb{\beta}\in B$ and the corresponding heat fluxes $\beta_0(u,\pmb{\beta})$ and $\beta_L(u,\pmb{\beta})$ by $\epsilon\mathbf{h}\in \mathbb{R}^{2n}$ leads to the heat fluxes $\beta_0(u_{\mathbf{h}},\pmb{\beta}+\epsilon\mathbf{h})$ and $\beta_L(u_{\mathbf{h}},\pmb{\beta}+\epsilon\mathbf{h})$, where $u_{\mathbf{h}}:=S(\pmb{\beta}+\epsilon\mathbf{h})$ is the solution of the perturbed initial boundary value problem
\begin{align}
\label{perturbedIBVP1}
  u_{\mathbf{h},t} &= (\alpha'(u_{\mathbf{h}})u_{\mathbf{h},x})_x,  &t \in I,\  x\in \Omega,\\
    \alpha'(u_{\mathbf{h}})u_{\mathbf{h},x} &= \beta_0(u_{\mathbf{h}},\pmb{\beta}+\epsilon\mathbf{h}), &t \in I,\  x=0,\\
   -\alpha'(u_{\mathbf{h}})u_{\mathbf{h},x} &= \beta_L(u_{\mathbf{h}},\pmb{\beta}+\epsilon\mathbf{h}), &t \in I,\  x=L,\\ \label{perturbedIBVP4}
    u_{\mathbf{h}}&=u_0, &t=0, x\in \bar{\Omega}.
\end{align}
Subtracting the original IBVP \eqref{IBVP1} - \eqref{IBVP4} from \eqref{perturbedIBVP1} - \eqref{perturbedIBVP4}, dividing the result by $\epsilon>0$ and considering the limit $\epsilon \to 0$ yield the assertion. 
\end{proof}

\begin{remark}
Note that, in comparison to \eqref{IBVP1}, the governing differential equation \eqref{gateauxPDE1} is linear with respect to its solution. Hence, the solvability theory (existence and uniqueness of solutions) is drastically simplified. Basic results on this matter, concerning initial boundary value problems with variable coefficients in H\"older function classes, were also formulated by Ladyzhenskaya et al., see \cite[Chapter 5]{ladyzenskaja1968linear} and especially Theorem 5.3 therein. We omit the detailed proof here, but point out that the assumptions of Theorem \ref{theoremClassicalSolution} are sufficient to guarantee the existence of a unique solution $w=\tilde{S}_{\pmb{\beta}}\mathbf{h}\in H^{2+l,1+l/2}(\bar{Q})\in \mathcal{U}$ and, hence, a well-defined operator \eqref{gateaux1}.
\end{remark}

\begin{theorem} Let $\pmb{\beta}\in B$ and $u:=S(\pmb{\beta})$. We assume that for any $v \in \mathcal{U}^*$ a unique solution $\varphi\in Z$, where $Z$ is a convenient function space, of the adjoint problem
\begin{align}
    \label{adjointProblem1}
    \varphi _t &=-\alpha'(u)\varphi_{xx}-v,  &t \in I,\  x\in \Omega,\\ \label{adjointProblem2}
    \alpha'(u)\varphi_x&=\beta_0'(u,\pmb{\beta})\varphi, &t \in I,\  x=0,\\ \label{adjointProblem3}
   -\alpha'(u)\varphi_x&=\beta_L'(u,\pmb{\beta})\varphi, &t \in I,\  x=L,\\ \label{adjointProblem4}
    \varphi &=0, &t=T, x\in \bar{\Omega},
\end{align}
exists. The adjoint operator \eqref{adjointOperatorOfLinearizedParameterToSolutionOperator}, or rather the adjoint state, can be computed explicitly by
\begin{align}
    \label{explicitComputationOfAdjointState}
    \tilde{S}^*_{\pmb{\beta}}v = \int_0^T \left(-\nabla\beta_0(u(t,0),\pmb{\beta})\ \varphi(t,0) -\nabla\beta_L(u(t,L),\pmb{\beta})\ \varphi(t,L)\right) \ \mathrm{d}t \ \ \in \mathbb{R}^{2n}.
\end{align}
\begin{proof}
For $\mathbf{h}\in \mathbb{R}^{2n}$ let $w:=\tilde{S}_{\pmb{\beta}}\mathbf{h}$. Regarding the inner product extension \eqref{innerproductExtensionOnU} the left hand side of \eqref{definitionAdjointOperator} is just
\begin{align}
    \label{firstStepOfAdjointProof}
    (w,v)_{L^2(Q)} = \int_I\int_{\Omega} w(t,x)v(t,x)\ \mathrm{d}x\ \mathrm{d}t.
\end{align}
Next we scrutinize the right hand side of \eqref{definitionAdjointOperator}. Multiplying \eqref{gateauxPDE1} by $\varphi\in Z$,  integrating over $I$ and $\Omega$ and using integration by parts two times leads to
\begin{align}
    \label{secondStepOfAdjointProof1}
    \int_I\int_{\Omega} w_t \varphi\ \mathrm{d}x\ \mathrm{d}t &= \int_I\int_{\Omega} (\alpha'(u)w)_{xx} \varphi\ \mathrm{d}x\ \mathrm{d}t\\
    &= \int_I\int_{\partial \Omega} (\alpha'(u)w)_x \varphi\ \mathrm{d}S\ \mathrm{d}t-\int_I\int_{\Omega} (\alpha'(u)w)_x \varphi_x\ \mathrm{d}x\ \mathrm{d}t \label{secondStepOfAdjointProof2}\\
     \label{secondStepOfAdjointProof3} &= \int_I\int_{\partial \Omega} (\alpha'(u)w)_x \varphi\ \mathrm{d}S\ \mathrm{d}t-\int_I\int_{\partial \Omega} (\alpha'(u)w) \varphi_x\ \mathrm{d}S\ \mathrm{d}t+ \int_I\int_{\Omega} (\alpha'(u)w) \varphi_{xx}\ \mathrm{d}x\ \mathrm{d}t \\
      &= \int_I\int_{\partial \Omega} (\alpha'(u)w)_x \varphi\ \mathrm{d}S\ \mathrm{d}t+ \int_I\int_{\Omega} (\alpha'(u)w) \varphi_{xx}\ \mathrm{d}x\ \mathrm{d}t
     \label{secondStepOfAdjointProof4}\\&\ -\int_I \left[(\alpha'(u)w) \varphi_x\right]_{x=L}-\left[(\alpha'(u)w) \varphi_x\right]_{x=0}\ \ \mathrm{d}t
\end{align}
By inserting the boundary conditions \eqref{gateauxPDE2} - \eqref{gateauxPDE3} into the first integral of \eqref{secondStepOfAdjointProof4}, we obtain
\begin{align}
    \label{thirdStepOfAdjointProof1}
    \int_I\int_{\Omega} w_t \varphi\ \mathrm{d}x\ \mathrm{d}t =& \int_I \ \left(-\beta_L'(u(t,L),\pmb{\beta})\ w(t,L)-\nabla  \beta_L(u(t,L),\pmb{\beta})\cdot \mathbf{h}\right)\ \varphi(t,L)\ \mathrm{d}t\\&+\int_I\left(-\beta_0'(u(t,0),\pmb{\beta})\ w(t,0)-\nabla  \beta_0(u(t,0),\pmb{\beta})\cdot \mathbf{h}\right)\ \varphi(t,0)\ \mathrm{d}t+\int_I\int_{\Omega} (\alpha'(u)w) \varphi_{xx}\ \mathrm{d}x\ \mathrm{d}t. \label{thirdStepOfAdjointProof2} \\ \label{thirdStepOfAdjointProof3}
    &-\int_I \left[(\alpha'(u)w) \varphi_x\right]_{x=L}-\left[(\alpha'(u)w) \varphi_x\right]_{x=0}\ \ \mathrm{d}t.
\end{align}
Regarding the time variable, we can partially integrate the left hand side of \eqref{thirdStepOfAdjointProof1} and use \eqref{gateauxPDE4}, which yields
\begin{align}
    \label{fourthSteopOfAdjointProof1}
    \int_I\int_{\Omega} w_t \varphi\ \mathrm{d}x\ \mathrm{d}t =& \int_{\Omega} w(T,x) \varphi(T,x)\ \mathrm{d}x-\int_I\int_{\Omega} w \varphi_t\ \mathrm{d}x\ \mathrm{d}t.
\end{align}
Using the last equality in \eqref{thirdStepOfAdjointProof1} and re-arranging the integrals, we arrive at
\begin{align}
    \label{fifthStepOfAdjointProof1}
    \int_I\int_{\Omega} w \ \left\{-\varphi_t-\alpha'(u)\varphi_{xx} \right\}\ \mathrm{d}x\ \mathrm{d}t\\ 
    \label{fifthStepOfAdjointProof2}
    +\int_I w(t,0)\ \left\{\left[-\alpha'(u) \varphi_x+\beta_0'(u,\pmb{\beta})\varphi\right]_{x=0}\right\}\ \mathrm{d}t\\ 
    \label{fifthStepOfAdjointProof3}
    +\int_I w(t,L)\ \left\{\left[\alpha'(u) \varphi_x+\beta_L'(u,\pmb{\beta})\varphi\right]_{x=L}\right\}\ \mathrm{d}t\\
    \label{fifthStepOfAdjointProof4}
    +\int_{\Omega} w(T,x) \varphi(T,x)\ \mathrm{d}x\\
    \label{fifthStepOfAdjointProof5}
    = \int_I \left(-\nabla\beta_0(u(t,0),\pmb{\beta})\ \varphi(t,0) -\nabla\beta_L(u(t,L),\pmb{\beta})\ \varphi(t,L)\right) \cdot \mathbf{h}\ \mathrm{d}t\\
    \label{fifthSteopOfAdjontProof6}
    \overset{\eqref{explicitComputationOfAdjointState}}{=} (\mathbf{h},\tilde{S}_{\pmb{\beta}}^*v)_2,
\end{align}
which is the right hand side of \eqref{definitionAdjointOperator}. If $\varphi$ solves the adjoint problem \eqref{adjointProblem1} - \eqref{adjointProblem4} the operator \eqref{explicitComputationOfAdjointState} fulfills \eqref{definitionAdjointOperator}, since its left hand side is given by \eqref{firstStepOfAdjointProof}.
\end{proof}
\end{theorem}

\begin{remark}
Note that we assumed that a unique solution $\varphi \in Z$ of the adjoint problem \eqref{adjointProblem1} - \eqref{adjointProblem4} exists without specifying the space $Z.$ In fact, we face a difficult issue here, because of the termination condition \eqref{adjointProblem4}. Since (parabolic) evolution equations are exponentially smoothing the initial condition in forward time, the calculation of the backward evolution equation is extremely ill-posed. 
There are regularization methods to solve such problems, e.g. the general method of quasireversibility, cf. \cite{lattes1967methode}, \cite{showalter1974final}. However, we do not want to deal with this topic in this article. Since the final condition \eqref{adjointProblem4} is homogeneous, we assume that the problem is well-posed and are satisfied with the time transformation $t=T-t$, yielding an initial boundary value problem, similar to \eqref{gateauxPDE1} - \eqref{gateauxPDE4}, with a unique solution $\varphi\in Z$. Additional smoothness conditions on $v\in \mathcal{U}^*$ lead
to the solution space $Z=H^{2+l,1+l/2}(\bar{Q})$, see Theorem \ref{theoremClassicalSolution}. For application purposes this would be possible by a modification of the observation operator
using $L^2-$ integrals to approximate point evaluations\setcounter{footnote}{0}\footnote{This can be achieved by so-called \textit{mollifiers}, see for example \cite{louis1996approximate}.} instead of point evaluations. Milder conditions lead to
$Z=W^{1,2,2}(I;V,V^*)\subset C (I;H)$ according to Theorem \ref{mainTheorem_Existence_Uniqueness}. Since the calculation of the adjoint state is only used for numerical purposes, all we need is a stable approximate solution achieved by some finite-difference scheme. 
\end{remark}

\section{Numerical experiments}\label{sec:implementationAndNumericalResults}

In this section we aim to solve the IHCP by implementing the minimization problem \eqref{minimizationTask} for the objective functional \eqref{objFunctional} using the Projected Quasi-Newton (PQN) method. Adapted to our problem, a PQN step is given by
\begin{align}
    \label{PQNstep}
    \pmb{\beta}^{(k+1)}=\mathcal{P}_B\left(\pmb{\beta}^{(k)}+\lambda_k \mathbf{p}_k\right),
\end{align}
where $\mathcal{P}_B:\mathbb{R}^{2n}\to B=[0,\beta_{max}]^{2n}$ is the metric projection enforcing the box-constraint. The step size $\lambda_k$ and the search direction $\mathbf{p}_k$ have to be determined accordingly.
As a stopping rule we use the discrepancy principle, which yields a finite stopping index $k_*$ when 
\begin{align}
    \label{diskrepanzPrinzip}
    \frac{f(\pmb{\beta}^{(k_*)})}{\|u^\delta\|_Y^2}\leq \rho\delta<\frac{f(\pmb{\beta}^{(k)})}{\|u^\delta\|_Y^2}
\end{align}
for all $k<k_*$ and fixed $\rho>1$.

We test the proposed method by means of synthetic measurement data with known noise level $\delta>0.$

\subsection{The PQN method}

The Projected Quasi-Newton method was proposed by Kim et al. (cf. \cite{kim2010tackling}), in order to solve box-constrained
optimization problems of the form \eqref{minimizationTask}.
There, the authors assume that the objective functional $f$ is twice continuously differentiable and strictly convex.
While Newton methods usually try to find the extrema of an objective functional by calculating the Hessian, Quasi-Newton methods try to approximate the Hessian via the knowledge of the gradient only, since the calculation of the Hessian is typically too expensive to compute at every iteration. The approximation of the Hessian matrix, or rather of its inverse, is updated at every iteration and is used to scale the gradient, yielding an approximated Newton step as a potential search direction. Probably the most famous updating procedure is given by the Broyden-Fletcher-Goldfarb-Shanno (BFGS) algorithm used in the associated Quasi-Newton BFGS method for unconstrained minimization problems.

Given the current approximation of the Hessian $H^{k}$, and vectors $\mathbf{s}_k$ and $\mathbf{g}_k$ by
\begin{align}
    \label{s_kAndy_k}
    \mathbf{s}_k=\pmb{\beta}^{(k+1)}-\pmb{\beta}^{(k)}, \text{ and } \mathbf{g}_k=\nabla f(\pmb{\beta}^{(k+1)})-\nabla f(\pmb{\beta}^{(k)}),
\end{align}
the Hessian updates are computed via the BFGS formula
\begin{align}
    \label{bfgsUpdate}
    H^{k+1} = H^{k}+\frac{\mathbf{g}_k \mathbf{g}_k^T}{\mathbf{g}_k^T \mathbf{s}_k}-\frac{H^{k}\mathbf{s}_k \mathbf{s}_k^T (H^k)^T}{\mathbf{s}_k^T H^{k}\mathbf{s}_k}.
\end{align}

Actually, since the search direction $\mathbf{p}_k$ is obtained by solving
\begin{align}
    \label{solveForSearchDirection}
    H^{k}\mathbf{p}_k=-\nabla f(\pmb{\beta}^{(k)}),
\end{align}
the calculation of the inverse of $H^{k}$ can be avoided by directly updating the inverse in each iteration step. This is achieved by applying the Sherman-Morrison formula which yields the update rule
\begin{align}\label{shermanMorrisonFormula}
    S^{k+1}=S^{k}+{\frac {(\mathbf {s} _{k}^{T }\mathbf {g} _{k}+\mathbf {g} _{k}^{T }S^{k}\mathbf {g} _{k})(\mathbf {s} _{k}\mathbf {s} _{k}^{T })}{(\mathbf {s} _{k}^{T }\mathbf {g} _{k})^{2}}}-{\frac {S^{k}\mathbf {g} _{k}\mathbf {s} _{k}^{T }+\mathbf {s} _{k}\mathbf {g} _{k}^{T }S^{k}}{\mathbf {s} _{k}^{T }\mathbf {g} _{k}}},
\end{align}
where $S^{k}$ is the inverse of $H^{k},$ resulting in the search direction
\begin{align}
    \label{searchDirectionForUnconstrained}
    \mathbf{p}_k = -S^k \nabla f (\pmb{\beta}^{(k)})
\end{align}
for the Quasi-Newton BFGS method for unconstrained problems. The method is known to be very efficient  and robust for smooth objective functionals. However, in \cite{curtis2015quasi} and \cite{xie2017convergence}, the authors point out that the BFGS method has a very good performance even for non-smooth functions, especially when combined with the Armijo-Wolfe line search.

The PQN method can be seen as an extension of the BFGS method to box-constrained optimization problems. The special feature here is the selection of the variables that can still be involved in the optimization process. Instead of \eqref{searchDirectionForUnconstrained} one computes the search direction
\begin{align}
    \label{searchDirectionForConstrained}
    \mathbf{p}_k=-\hat{S}^{k}\nabla f(\pmb{\beta}^{(k)}),
\end{align}
where
\begin{align}
    \label{Shat}
    \hat{S}^{k}=\begin{cases}S^k_{ij},\ \ \ \ &\text{if } i,j \notin \ I^k_1\cup I^k_2,\\ 0, \ \ \ \ &\text{otherwise}\end{cases}
\end{align}
keeps track of the parameter variables that are fixed or free. The fixed variables are collected via the index set $I^k_1\cup I^k_2$, to be defined in the following. The index set
$$I_1^k=\left\{i: \beta_i^{(k)}=0\wedge \left[\nabla f(\pmb{\beta}^{(k)})\right]_i>0,\ \text{ or } \beta_i^{(k)}=\beta_{max} \wedge \left[\nabla f(\pmb{\beta}^{(k)})\right]_i<0\right\}$$
represents fixed variables on the boundaries of the box $B$, so-called active variables, which are useless for further iterations since they do not decrease the function value any further. Thus, for active variables not contained in $I_1^k$ the next iteration for, e.g., the first case $\beta_i^{(k)}=0$ might lead to $\beta_i^{(k+1)}>0$ and 
\begin{align}
\label{descentCondition}
    f(\pmb{\beta}^{(k+1)})<f(\pmb{\beta}^{(k)}).
\end{align} 
Let
$$\bar{S}^{k}=\begin{cases}S^k_{ij},\ \ \ \ &\text{if } i,j \notin I^k_1,\\ 0, \ \ \ \ &\text{otherwise}.\end{cases}$$
Since the PQN method is a gradient-scaling method, the descent condition
\eqref{descentCondition} still might not be satisfied.\setcounter{footnote}{0}\footnote{This is possible, because the descent direction component is not $\left[\nabla f(\pmb{\beta}^{(k)})\right]_i$ but $\left[\bar{S}^k\nabla f(\pmb{\beta}^{(k)}\right]_i$.}
For that reason, it is important to handle active variables left by $I^k_1$, which still need to be fixed. This is done with the index set
$$I_2^k=\left\{i: \beta_i^{(k)}=0\wedge \left[\bar{S}^k \nabla f(\pmb{\beta}^{(k)})\right]_i>0,\ \text{ or } \beta_i^{(k)}=\beta_{max} \wedge \left[\bar{S}^k \nabla f(\pmb{\beta}^{(k)})\right]_i<0\right\}.$$
 For further details we refer the interested reader to \cite{kim2010tackling}.

Given the search direction \eqref{searchDirectionForConstrained}, in order to compute a step size $\lambda_k$ for the PQN step \eqref{PQNstep}, a line search algorithm is needed. In this article, we implement the backtracking line search algorithm based on the Armijo condition, where a step size $\lambda_k>0$ is accepted when
\begin{align}
\label{armijoCond}
    f(\pmb{\beta}^{(k)})-f\left(\mathcal{P}_B(\pmb{\beta}^{(k)}+\lambda_k \mathbf{p}_k)\right) \geq -c\lambda_k \nabla f(\pmb{\beta}^{(k)})^T\mathbf{p}_k
\end{align}
is fulfilled for some constant $c\in(0,1).$ If $\lambda_k$ does not satisfy \eqref{armijoCond}, the step size is reduced repeatedly to $\lambda_k:=\tau \lambda_k$ for a $\tau\in(0,1)$. We use $c=\tau=0.5$, as it was proposed in the original work by Armijo in \cite{armijo1966minimization}.

To start with the PQN method, we choose $S^{0}=I$, i.e. the first iteration corresponds to the classical gradient descent method. In the further iteration steps the method builds up curvature information via \eqref{shermanMorrisonFormula}, which accelerates the minimization process significantly since we involve second derivative knowledge.

\subsection{The experimental setting}

To fully describe the objective functional \eqref{objFunctional} for numerical investigations, we need to 
\begin{itemize}
    \item[(1)] fix the setting of the IBVP \eqref{IBVP1} - \eqref{IBVP4}, and
    \item[(2)] generate noisy measurement data $u^{\delta}.$
\end{itemize} 

First of all, we do not use the non-dimensionalized version of the IBVP. To remind the Accelerated Cooling (ACC) process of hot-rolled heavy plates made of steel we try to keep the experiment close to real values, but still neglect most of the SI units later on. To this end we choose a plate with a thickness of $50\  \si{mm}$, i.e. $L=0.05\ [\si{m}]$ leading to $\Omega=[0,0.05]$. For the time interval $I=[0,T]$ we set $T=30\ [\si{s}]$, which is a typical total cooling time in the underlying ACC process. By defining temperature dependent material parameters $k$ and $C$ in \eqref{oldPde}, we transform temperature to enthalpy by \eqref{enthalpyDef}. Hence, by assuming a homogeneous initial temperature of approximately $1100 \si{K}$ we choose $u(0,x) = u_0(x) = 5.5\times 10^9$. (We note that the understanding of the conversion is not relevant at this point.) We also use this value for the upper bound of enthalpy $u_{max}$. Furthermore, we compute the enthalpy-dependent function $\alpha'(u)\in \mathcal{C}^1([0,u_{max}])$, cf. \eqref{alphaStrichGleichTLF}, which is plotted in Figure \ref{fig:alphaPrime}. 
\begin{center}
\def\svgwidth{400pt}
\begingroup%
  \makeatletter%
  \providecommand\color[2][]{%
    \errmessage{(Inkscape) Color is used for the text in Inkscape, but the package 'color.sty' is not loaded}%
    \renewcommand\color[2][]{}%
  }%
  \providecommand\transparent[1]{%
    \errmessage{(Inkscape) Transparency is used (non-zero) for the text in Inkscape, but the package 'transparent.sty' is not loaded}%
    \renewcommand\transparent[1]{}%
  }%
  \providecommand\rotatebox[2]{#2}%
  \newcommand*\fsize{\dimexpr\f@size pt\relax}%
  \newcommand*\lineheight[1]{\fontsize{\fsize}{#1\fsize}\selectfont}%
  \ifx\svgwidth\undefined%
    \setlength{\unitlength}{1440bp}%
    \ifx\svgscale\undefined%
      \relax%
    \else%
      \setlength{\unitlength}{\unitlength * \real{\svgscale}}%
    \fi%
  \else%
    \setlength{\unitlength}{\svgwidth}%
  \fi%
  \global\let\svgwidth\undefined%
  \global\let\svgscale\undefined%
  \makeatother%
  \begin{picture}(1,0.48072917)%
    \lineheight{1}%
    \setlength\tabcolsep{0pt}%
    \put(0,0){\includegraphics[width=\unitlength,page=1]{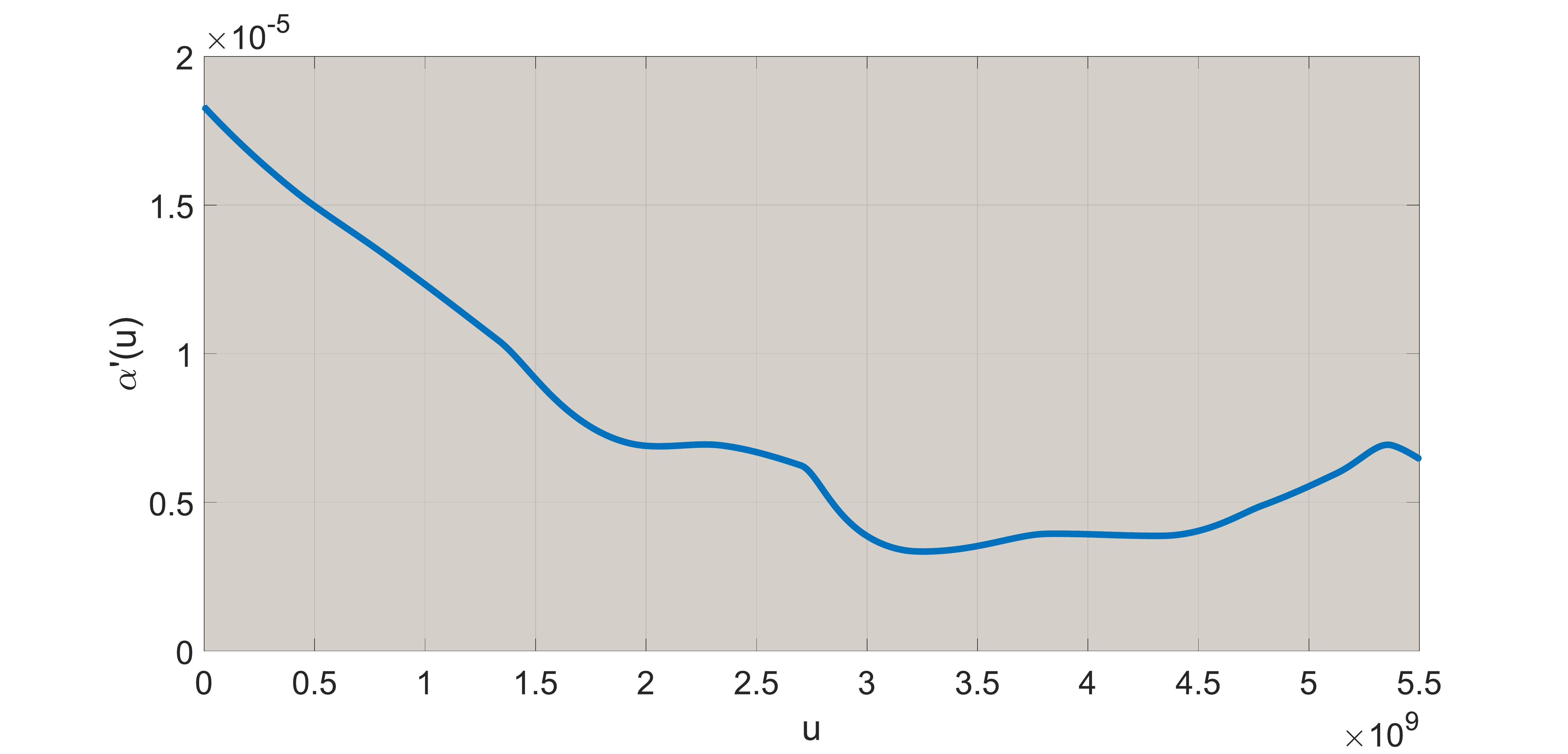}}%
  \end{picture}%
\endgroup%

\captionof{figure}{Thermal diffusivity  $\alpha'(u)$ in terms of enthalpy $u$}
\label{fig:alphaPrime}
\end{center}

So far, the governing evolution equation \eqref{IBVP1} and the initial condition $\eqref{IBVP4}$ are fixed.

We generate measurement data by putting the exact heat fluxes $\beta_0^{ex}(u), \beta_L^{ex}(u)\in \mathcal{C}^1([0,u_{max})]$ (cf. Figure \ref{fig:beta0Lexact}) in the boundary conditions \eqref{IBVP2} - \eqref{IBVP3} and numerically computing the associated solution $u^{ex}\in \mathcal{U}$ of the IBVP,

\begin{center}
\def\svgwidth{\textwidth}
\begingroup%
  \makeatletter%
  \providecommand\color[2][]{%
    \errmessage{(Inkscape) Color is used for the text in Inkscape, but the package 'color.sty' is not loaded}%
    \renewcommand\color[2][]{}%
  }%
  \providecommand\transparent[1]{%
    \errmessage{(Inkscape) Transparency is used (non-zero) for the text in Inkscape, but the package 'transparent.sty' is not loaded}%
    \renewcommand\transparent[1]{}%
  }%
  \providecommand\rotatebox[2]{#2}%
  \newcommand*\fsize{\dimexpr\f@size pt\relax}%
  \newcommand*\lineheight[1]{\fontsize{\fsize}{#1\fsize}\selectfont}%
  \ifx\svgwidth\undefined%
    \setlength{\unitlength}{1440bp}%
    \ifx\svgscale\undefined%
      \relax%
    \else%
      \setlength{\unitlength}{\unitlength * \real{\svgscale}}%
    \fi%
  \else%
    \setlength{\unitlength}{\svgwidth}%
  \fi%
  \global\let\svgwidth\undefined%
  \global\let\svgscale\undefined%
  \makeatother%
  \begin{picture}(1,0.46979167)%
    \lineheight{1}%
    \setlength\tabcolsep{0pt}%
    \put(0,0){\includegraphics[width=\unitlength,page=1]{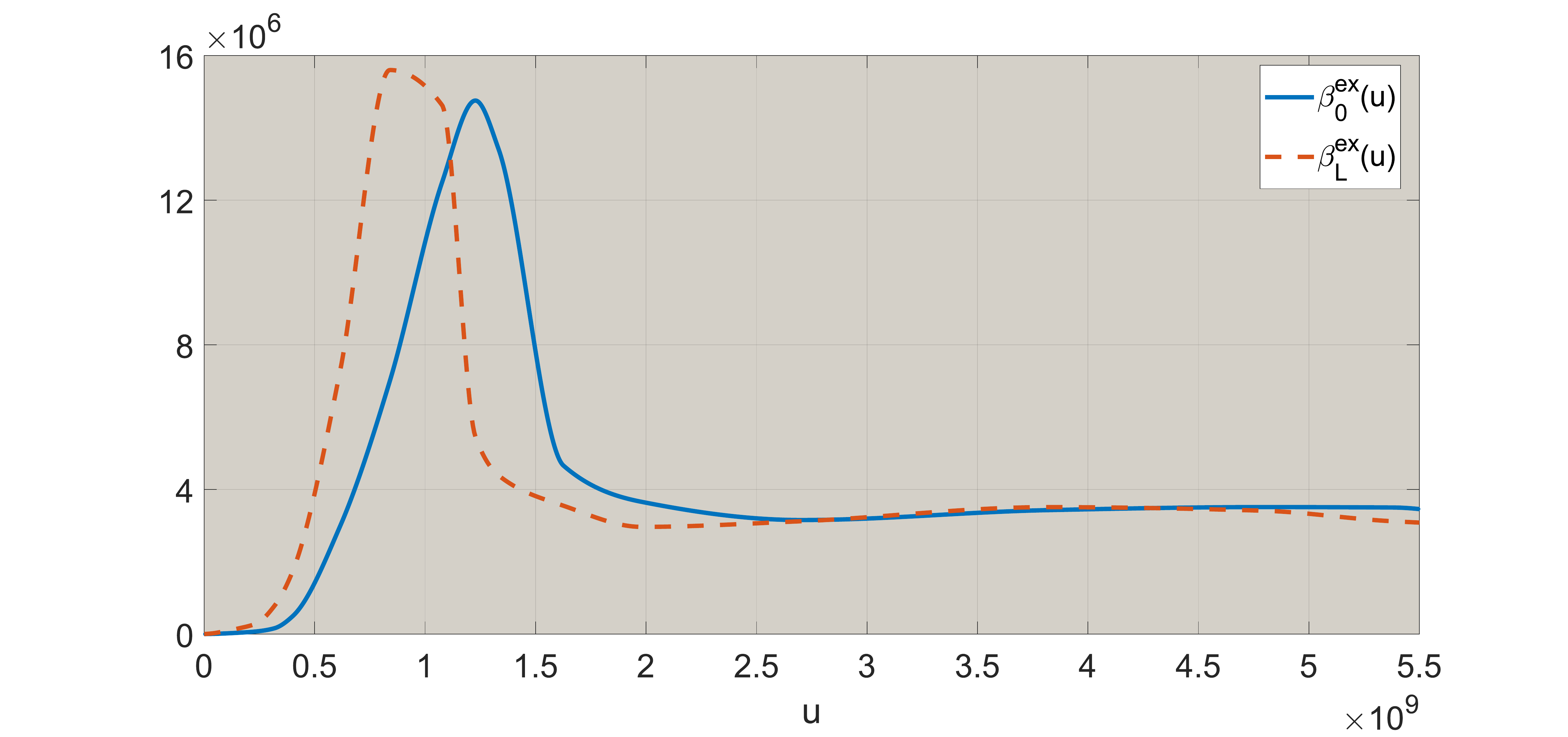}}%
  \end{picture}%
\endgroup%

\captionof{figure}{Exact heat fluxes to simulate $u^{\delta}$}
\label{fig:beta0Lexact}
\end{center}

which we obtain on a finite space and time grid by a stable Finite Difference Method (FDM), see Figure \ref{fig:uexact}. 

\begin{center}
\def\svgwidth{\textwidth}
\begingroup%
  \makeatletter%
  \providecommand\color[2][]{%
    \errmessage{(Inkscape) Color is used for the text in Inkscape, but the package 'color.sty' is not loaded}%
    \renewcommand\color[2][]{}%
  }%
  \providecommand\transparent[1]{%
    \errmessage{(Inkscape) Transparency is used (non-zero) for the text in Inkscape, but the package 'transparent.sty' is not loaded}%
    \renewcommand\transparent[1]{}%
  }%
  \providecommand\rotatebox[2]{#2}%
  \newcommand*\fsize{\dimexpr\f@size pt\relax}%
  \newcommand*\lineheight[1]{\fontsize{\fsize}{#1\fsize}\selectfont}%
  \ifx\svgwidth\undefined%
    \setlength{\unitlength}{1440.0719bp}%
    \ifx\svgscale\undefined%
      \relax%
    \else%
      \setlength{\unitlength}{\unitlength * \real{\svgscale}}%
    \fi%
  \else%
    \setlength{\unitlength}{\svgwidth}%
  \fi%
  \global\let\svgwidth\undefined%
  \global\let\svgscale\undefined%
  \makeatother%
  \begin{picture}(1,0.46979166)%
    \lineheight{1}%
    \setlength\tabcolsep{0pt}%
    \put(0,0){\includegraphics[width=\unitlength,page=1]{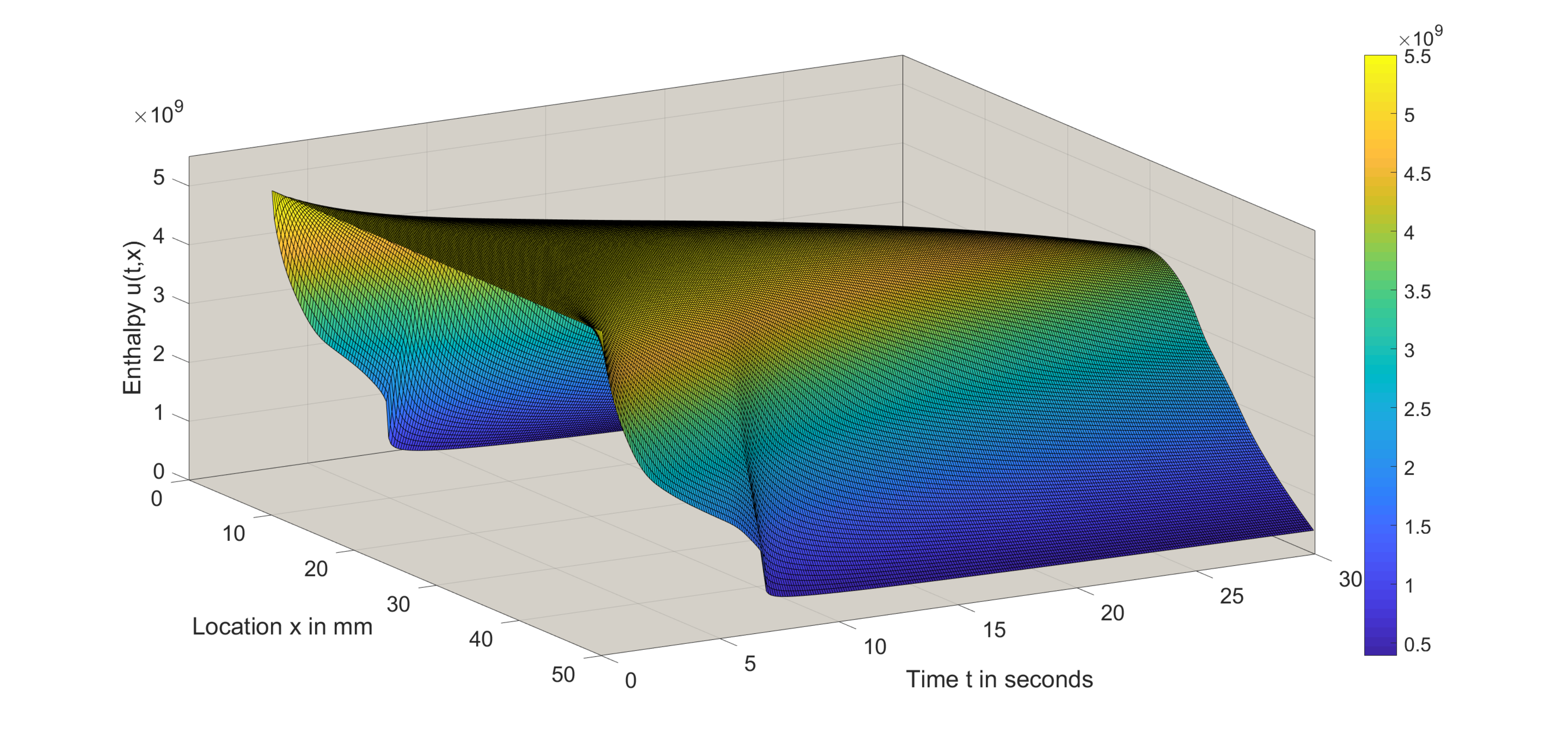}}%
  \end{picture}%
\endgroup%

\captionof{figure}{Approximate solution of $u^{ex}$}
\label{fig:uexact}
\end{center}

Note that the heat conduction inside of the domain $\Omega$ is solely driven by $\alpha'$, while the heat convection on the boundaries is driven by the heat fluxes $\beta_0^{ex}$ and $\beta_L^{ex}$, that carry the information of the Leidenfrost effect. Coming from high enthalpy values the starting of the peaks in Figure \ref{fig:beta0Lexact} simulates the steam layer collapse, leading to an increased heat flux.

Concerning the measuring procedure we model the insertion of $d=5$ thermocouples into the heavy plate at positions
\begin{align}
\label{thermocouplePositions}
    x \in (x_1=0.002,\  x_2=0.01,\  x_3=0.025,\  x_4=0.04,\  x_5=0.048)^T,
\end{align}
which is exactly at the core and $2\si{\mm}$ and $10\si{\mm}$ beneath both surfaces. During the cooling process, every thermocouple then records the enthalpy value every $0.1$ seconds, yielding $m=300$ time samples at
\begin{align}
\label{thermocoupleTimes}
    t\in (t_1=0.1, \ t_2=0.2,\  \dots,\  t_{300}=T=30)^T.
\end{align}
Hence, we obtain an observed enthalpy matrix $Qu^{ex} \in Y=\mathbb{R}^{5\times 300}$, cf. \eqref{specifiedObservationOperator}. Simulating measurement errors and assuming a sensor accuracy in the range of $\pm 0.5\  \si{K}$, we a add uniformly distributed random numbers in the interval $[-2,2]\times  10^6$
to each component of $Qu^{ex}$, leading to noisy data $u^{\delta}$. Calculating the noise level $\delta>0$ using the Frobenius-Norm via
\begin{align}
\label{noiseExecution}
    \frac{1}{2}\|Qu^{ex}-u^{\delta}\|_Y^2\leq \delta\|u^\delta\|^2_Y
\end{align}
yields different results for each program sweep, where the noise lavel not exceed $\delta:=6.65\times 10^{-8}$.

\subsection{Numerical results}

Provided $\alpha'$ and $u_0$ from the experimental setting, the forward operator \eqref{forwardOperatorInIntroduction} is fully described. We determine the heat flux parameter $\pmb{\beta}\in B$ and the associated heat fluxes/ PCHIP interpolants $\beta_0(u,\pmb{\beta}),\ \beta_L(u,\pmb{\beta})\in \mathcal{C}^1([0,u_{max}])$ using the data $u^{\delta}$ and compare the result with the exact heat fluxes used for the simulation above. 

For $\pi_n$ in \eqref{definitionPIn} we choose an equidistant partition of the interval $[0,u_{max}]$ for $n=20,$ i.e. $u_i=u_{max}\cdot\left(\frac{i-1}{n-1}\right)$ for $i=1,\dots,n$.
As for the box-constraint $B=[0,\beta_{max}]^{2n}$ we set $\beta_{max}=16\times 10^6.$

In order to start with a PQN step \eqref{PQNstep}, we initialize $\pmb{\beta}^{(0)}=\mathbf{0}\in\mathbb{R}^{2n}.$ Hence, assuming that we have no a-priori information about the exact solution of the IHCP. Given the parameter $\pmb{\beta}^{(k)}$ for $k=0,1,\dots,k_*$, we perform the process chain \eqref{processChain} of the parameter-to-solution operator $S$. Again, we obtain an enthalpy solution by solving the IBVP by a stable FDM, where we use another time and space grid in order to avoid an inverse crime (\cite{colton2019inverse}) of simulating and inverting the data using the same numerical approximations. At this point, we also mention that the exact heat fluxes in Figure \ref{fig:beta0Lexact} are piecewise cubic interpolants on a completely different partition than $\pi_n$. Otherwise, the inversion process would really invert exactly the approximate numerical forward model.

In the continuous setting, the solution $S(\pmb{\beta}^{(k)})$ would then be applied to the observation operator $\mathcal{Q}$ defined by \eqref{specifiedObservationOperator}, which is now specified by \eqref{thermocouplePositions} and \eqref{thermocoupleTimes}. Since we only have an approximate solution, we just interpolate the enthalpy values with respect to the time and space grid, if necessary. This way, we get an approximation of the enthalpy matrix $F(\pmb{\beta}^{(k)})$ which we can compare with the measurement data via the objective functional \eqref{objFunctional}.

To actually perform \eqref{PQNstep} at iteration $k$, we compute \eqref{gradientOfObjectiveFunctional} by\setcounter{footnote}{0}\footnote{This follows from \eqref{adjointOperatorOfFequalsSPrimeStarQStar} and \eqref{adjungierteSindGleich}.} \begin{align}
    \label{approximationOfTheGradient}
    \nabla f(\pmb{\beta}^{(k)}) = \tilde{S}_{\pmb{\beta}^{(k)}}^*\circ \mathcal{Q}^*\left(F(\pmb{\beta}^{(k)}-u^{\delta}\right),
\end{align} with $\mathcal{Q}^*$ and $\tilde{S}_{\pmb{\beta}^{(k)}}^*$ from \eqref{adjointOfObservationOperator} and \eqref{adjointOperatorOfLinearizedParameterToSolutionOperator}, respectively. The latter operator is computed by solving the adjoint problem \eqref{adjointProblem1} - \eqref{adjointProblem4} for $v=\mathcal{Q}^*\left(F(\pmb{\beta}^{(k)}-u^{\delta}\right)$ with a stable FDM, yielding a solution $\varphi$ which is used to compute the integral \eqref{explicitComputationOfAdjointState} by the trapezoidal summation rule. We note that the gradients $\nabla\beta_0$ and $\nabla\beta_L$ of the interpolants (with respect to $\pmb{\beta}^{(k)}$) can be computed analytically if the interpolation method is known. For the PCHIP interpolation method, we provide the necessary formulas in the appended Section \ref{pchipsGradientenberechnung}.

As mentioned, by choosing $S^0=I$ in \eqref{searchDirectionForUnconstrained}, the first PQN step corresponds to a gradient descent step. From $k=1$ on, the calculated gradients \eqref{approximationOfTheGradient} are scaled by the inverse Hessian approximations, which are updated in every iteration by \eqref{shermanMorrisonFormula}. Hence, we obtain the search direction \eqref{searchDirectionForConstrained} by keeping track of the free and fixed variables. Then the backtracking line search algorithm \eqref{armijoCond} is applied to find an appropriate step size $\lambda_k>0$ in order to finally compute the PQN step. 

This procedure is repeated until the discrepancy principle \eqref{diskrepanzPrinzip} is satisfied. Taking discretization errors due to several approximations into account, we choose $\rho=2$ to stop the iterations in time, i.e. $\rho\delta=1.33\times 10^{-7}.$

We test the performance of the PQN method in comparison to the attenuated Landweber method
\begin{align}
    \label{LandweberStep}
    \pmb{\beta}^{(k+1)}=\pmb{\beta}^{(k)}-\lambda \nabla f(\pmb{\beta}^{(k)})
\end{align}
with constant damping factor $\lambda>0,$ which is known to be very stable but slow. Because we do not expect that the iteration will stop via \eqref{diskrepanzPrinzip}, we set the maximum iteration count to $K=10,000.$

The numerical results are included in Figure \ref{fig:heatFluxesResults}.

\begin{center}
\def\svgwidth{\textwidth}
\begingroup%
  \makeatletter%
  \providecommand\color[2][]{%
    \errmessage{(Inkscape) Color is used for the text in Inkscape, but the package 'color.sty' is not loaded}%
    \renewcommand\color[2][]{}%
  }%
  \providecommand\transparent[1]{%
    \errmessage{(Inkscape) Transparency is used (non-zero) for the text in Inkscape, but the package 'transparent.sty' is not loaded}%
    \renewcommand\transparent[1]{}%
  }%
  \providecommand\rotatebox[2]{#2}%
  \newcommand*\fsize{\dimexpr\f@size pt\relax}%
  \newcommand*\lineheight[1]{\fontsize{\fsize}{#1\fsize}\selectfont}%
  \ifx\svgwidth\undefined%
    \setlength{\unitlength}{1226.99998454bp}%
    \ifx\svgscale\undefined%
      \relax%
    \else%
      \setlength{\unitlength}{\unitlength * \real{\svgscale}}%
    \fi%
  \else%
    \setlength{\unitlength}{\svgwidth}%
  \fi%
  \global\let\svgwidth\undefined%
  \global\let\svgscale\undefined%
  \makeatother%
  \begin{picture}(1,0.6601467)%
    \lineheight{1}%
    \setlength\tabcolsep{0pt}%
    \put(0,0){\includegraphics[width=\unitlength,page=1]{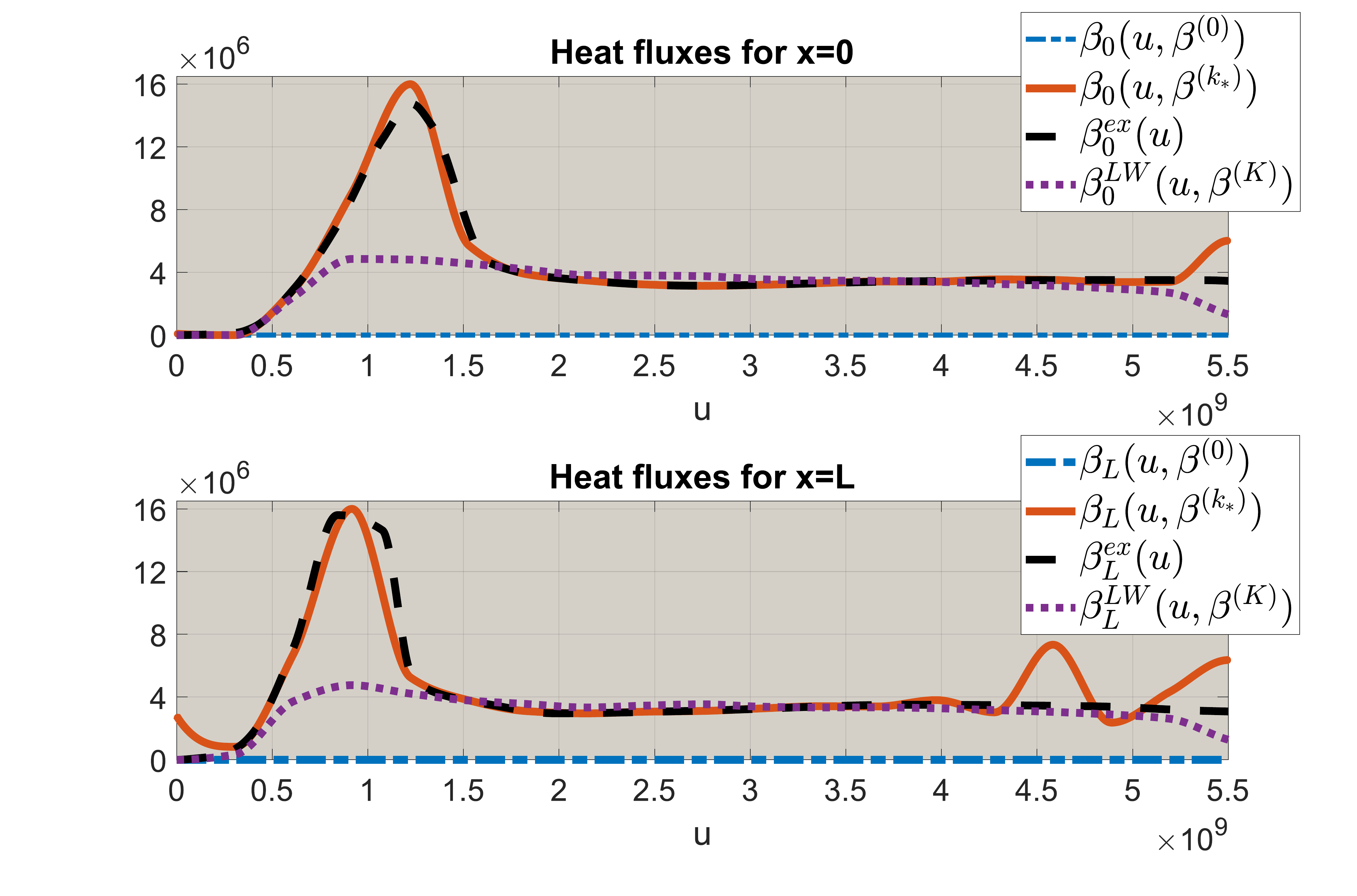}}%
  \end{picture}%
\endgroup%

\captionof{figure}{Numerical results and comparison of exact heat fluxes to PCHIP interpolants given (a) the initial parameter $\pmb{\beta}^{(0)}$, (b) the optimized parameter $\pmb{\beta}^{(k_*)}$ via the PQN method, (c) the optimized parameter $\pmb{\beta}^{(K)}$ via the attenuated Landweber method with $K=10,000$ iterations}
\label{fig:heatFluxesResults}
\end{center}

\begin{center}
\def\svgwidth{350pt}
\begingroup%
  \makeatletter%
  \providecommand\color[2][]{%
    \errmessage{(Inkscape) Color is used for the text in Inkscape, but the package 'color.sty' is not loaded}%
    \renewcommand\color[2][]{}%
  }%
  \providecommand\transparent[1]{%
    \errmessage{(Inkscape) Transparency is used (non-zero) for the text in Inkscape, but the package 'transparent.sty' is not loaded}%
    \renewcommand\transparent[1]{}%
  }%
  \providecommand\rotatebox[2]{#2}%
  \newcommand*\fsize{\dimexpr\f@size pt\relax}%
  \newcommand*\lineheight[1]{\fontsize{\fsize}{#1\fsize}\selectfont}%
  \ifx\svgwidth\undefined%
    \setlength{\unitlength}{1227bp}%
    \ifx\svgscale\undefined%
      \relax%
    \else%
      \setlength{\unitlength}{\unitlength * \real{\svgscale}}%
    \fi%
  \else%
    \setlength{\unitlength}{\svgwidth}%
  \fi%
  \global\let\svgwidth\undefined%
  \global\let\svgscale\undefined%
  \makeatother%
  \begin{picture}(1,0.6601467)%
    \lineheight{1}%
    \setlength\tabcolsep{0pt}%
    \put(0,0){\includegraphics[width=\unitlength,page=1]{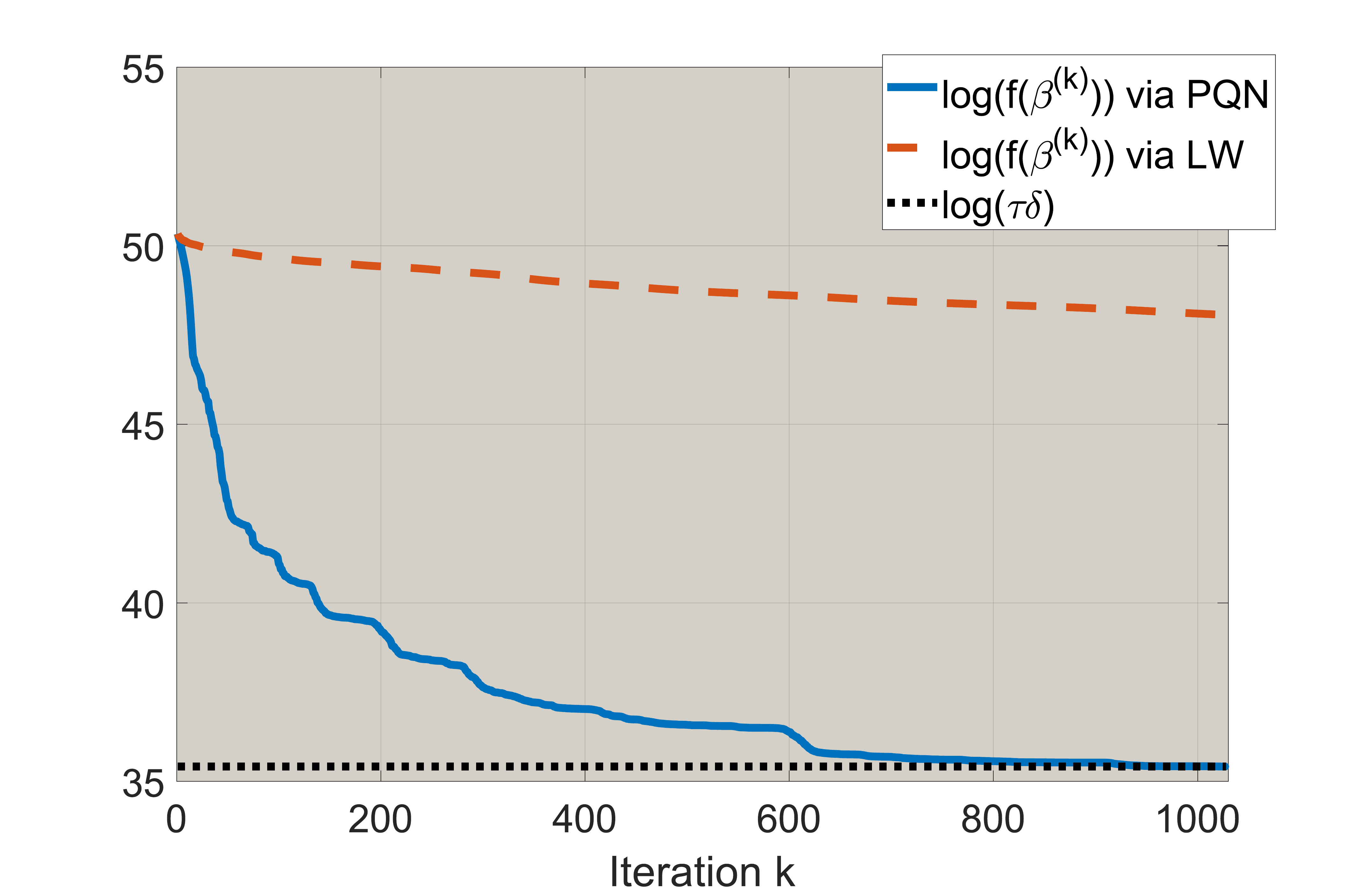}}%
  \end{picture}%
\endgroup%

\captionof{figure}{Logarithm of the residuals for the first $k_*=1028$ iterations and the stopping criteria threshold}
\label{fig:logResiduals}
\end{center}

In Figure \ref{fig:logResiduals}, plotting the logarithm of the residuals \eqref{objFunctional}, we observe that the PQN method stops at iteration $k_*=1028$. The attenuated Landweber method does not match the discrepancy principle, even after $K=10,000$ iterations. Increasing the damping factor $\lambda>0$ in \eqref{LandweberStep} to speed up the Landweber process, we experienced numerical instabilities. Instead of fixing $\lambda$ we even performed the backtracking line search algorithm, but did not realize any significant improvements for the Landweber method. The PQN method is superior due to the involvement of curvature information. Figure \ref{fig:heatFluxesResults} shows, that the interpolants $\beta_0(u,\pmb{\beta}^{(k_*)})$ and $\beta_L(u,\pmb{\beta}^{(k_*)})$ are a very good fit to the exact heat fluxes. The oscillations on the right-hand side can be explained due to the noisy data and a bigger number of approximation steps that have to be performed for the PQN method. Because the main information about the Leidenfrost effect lies in the peaks on the left-hand side, these errors are negligible, especially, when considering the significant reduction of the computational time.

\section{Conclusion}

Initially given an implicitly defined forward operator, we were able to numerically determine enthalpy-dependent
heat fluxes from internal measurements in a nonlinear inverse heat convection problem. To this end, we analyzed two different approaches in showing existence and uniqueness to solutions of the underlying IBVP. By proposing a convenient parametrization method using PCHIPs, we overcame the implicitness of the problem and defined a decoupled parameter-to-solution operator $S.$ Applying the method, we did not limit the scope of the inverse problem by prescribing the structure of the solution too much. Hence, we were able to identify continuously differentiable heat fluxes without any a-priori information. We derived the necessary ingredients in order to compute the gradient of the objective functional. Finally, we utilized the PQN method in order to accelerate the computations drastically compared to the attenuated Landweber method, while still obtaining remarkable numerical results.


\begin{appendix}
\section{PCHIPs - Piecewise Cubic Hermite Interpolating Polynomials}\label{pchips}
The interpolation method of Fritsch and Carlson was published in 1980 in \cite{fritschCarlson}. The method is included in several software packages for programming languages such as Matlab, R and Python. In the following we discuss the construction details of PCHIPs, their approximation property and the calculation of their gradients with respect to the function values.

\subsection{Construction details}\label{pchipsKonstruktion}
Let $I=[a,b]\in \mathbb{R}$, $\mathbb{N}\ni n>1$ and $$\pi_n:a=x_1< x_2< \ldots < x_n=b$$ be a partition of the interval $I$.
Furthermore, let $\{ f_i:i=1,2, \ldots, n  \}$ be the given set of data points on the partition knots
$\{ x_i:i=1,2, \ldots, n  \}$. The goal is to construct a piecewise cubic function $p(x) \in \mathcal{C}^1[I]$ on $\pi_n$, which is monotone on every subinterval $I_i=[x_i, x_{i+1}]$, and such that
\begin{align*}
p(x_i)=f_i, \qquad i=1,2,\ldots, n.
\end{align*}
The monotonicity requirement on the so-called PCHIP interpolant $p(x)$ then guarantees, that no overshoots or increased oscillations occur in the graph of the function, cf. Figure \ref{fig:pchipVsSpline}.

For $x\in I_i$, the function $p(x)$ can be represented as a cubic polynomial by
\begin{align*}
p(x)=f_i H^i_1(x)+f_{i+1}H_2^i(x) + d_iH_3^i(x)+d_{i+1} H_4^i(x),
\end{align*}
where $d_j=p'(x_j)$ for $j=i,i+1$. 
The functions $H_k^i(x)$, $k=1,\dots,4$, are the classical cubic Hermite basis functions for the interval $I_i$ given by
\begin{align}
H^i_1(x)=\phi((x_{i+1}-x)/h_i), \qquad &H^i_2(x)=\phi((x-x_i)/h_i),\notag \\
H^i_3(x)=-h_i \psi ((x_{i+1}-x)/h_i), \qquad &H^i_4(x)=h_i \psi((x-x_i)/h_i), \label{lab:hermiteBasisfunctions}
\end{align}
where $h_i=x_{i+1}-x_i$, $\phi(t)=3t^2-2t^3$ and $\psi(t)=t^3-t^2$, cf. Figure \ref{fig:hermiteBasisfunctions}.

\begin{figure}[h]
    \centering
        \def\svgwidth{\textwidth}
    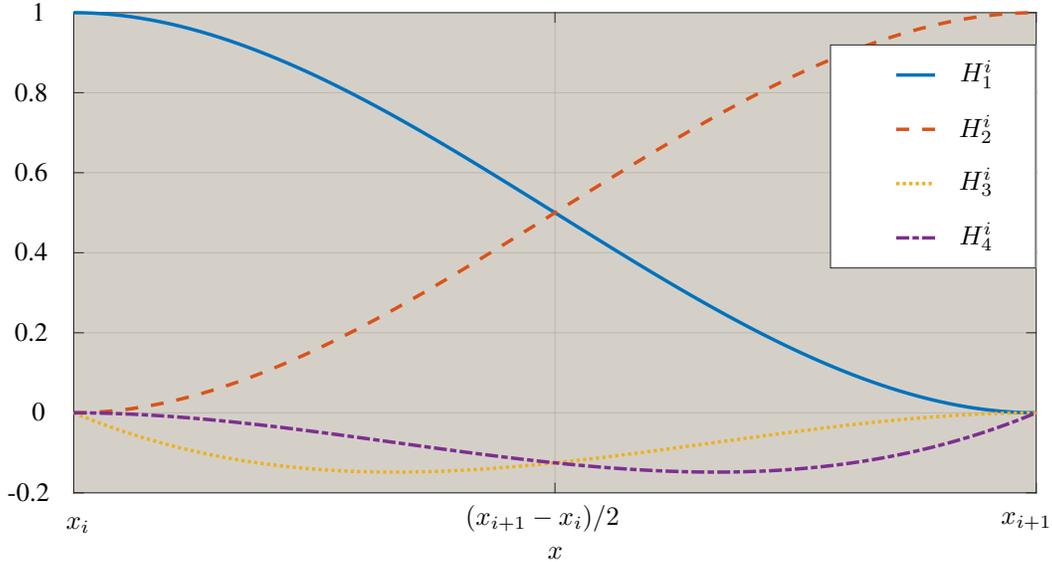
    \captionof{figure}{Hermite basis functions \eqref{lab:hermiteBasisfunctions} on the interval $I_i=[x_i,x_{i+1}]$ for $h_i := 1$}
    \label{fig:hermiteBasisfunctions}
\end{figure}

The construction of the piecewise cubic interpolant consists essentially in the approximation of the derivatives $d_1,\dots,d_n$, for which there are various possibilities, see \cite{fritschCarlson}.
If we consider an equidistant partition $\pi_n$ of the interval $I$ with $h_i=h_j=:h$ for all $i,j=1,2,\ldots, n$, we can calculate the derivatives $d_1, \ldots, d_n$ in this article by
\begin{align*}
d_1=\frac{3}{2}\Delta_1-\frac{1}{2}\Delta_2, \qquad \ldots \qquad d_{k+1}=\frac{2|\Delta_k\Delta_{k+1}|}{\Delta_k + \Delta_{k+1}}, \qquad \ldots \qquad d_n=\frac{3}{2}\Delta_{n-1}-\frac{1}{2}\Delta_{n-2}, 
\end{align*}
where $\Delta_i=\frac{f_{i+1}-f_i}{h}$ for all $i=1,\ldots, n$. In the original work by Fritsch and Carlson, it is assumed that the data points are monotone throughout the whole set, i.e. $f_i\leq f_{i+1}$ or $f_i\geq f_{i+1}$ for all $i=1,\dots,n-1$. However, this requirement is no longer necessary provided the modification
\vspace{2mm}\\ \
If $\sgn(\Delta_k\Delta_{k+1})<0$, set $d_{k+1}:=0.$
\vspace{2mm}\\

Hence, the PCHIP approach proposed in this paper in order to interpolate unknown functions respects the monotonicity of the function values, i.e. for $i=1,\dots,n-1$ we have $f_i\leq p(x)\leq f_{i+1}$ or $f_i\geq p(x)\geq f_{i+1}$ for all $x\in [x_i,x_{i+1}]$.
Thus, the PCHIP interpolant $p(x)$, in comparison to the interpolant of the cubic spline interpolation, does not allow overshoots or oscillations. This is important because we use a projection method in this paper and the searched functions must satisfy certain box constraints. In Figure \ref{fig:pchipVsSpline} we use a simple example to show that the spline interpolant violates the box constraints on both the top and bottom side. Note that the cubic spline interpolant is in $\mathcal{C}^2$, while the PCHIP interpolant is only in $\mathcal{C}^1$.   

\begin{figure}[h]
    \centering
        \def\svgwidth{\textwidth}
    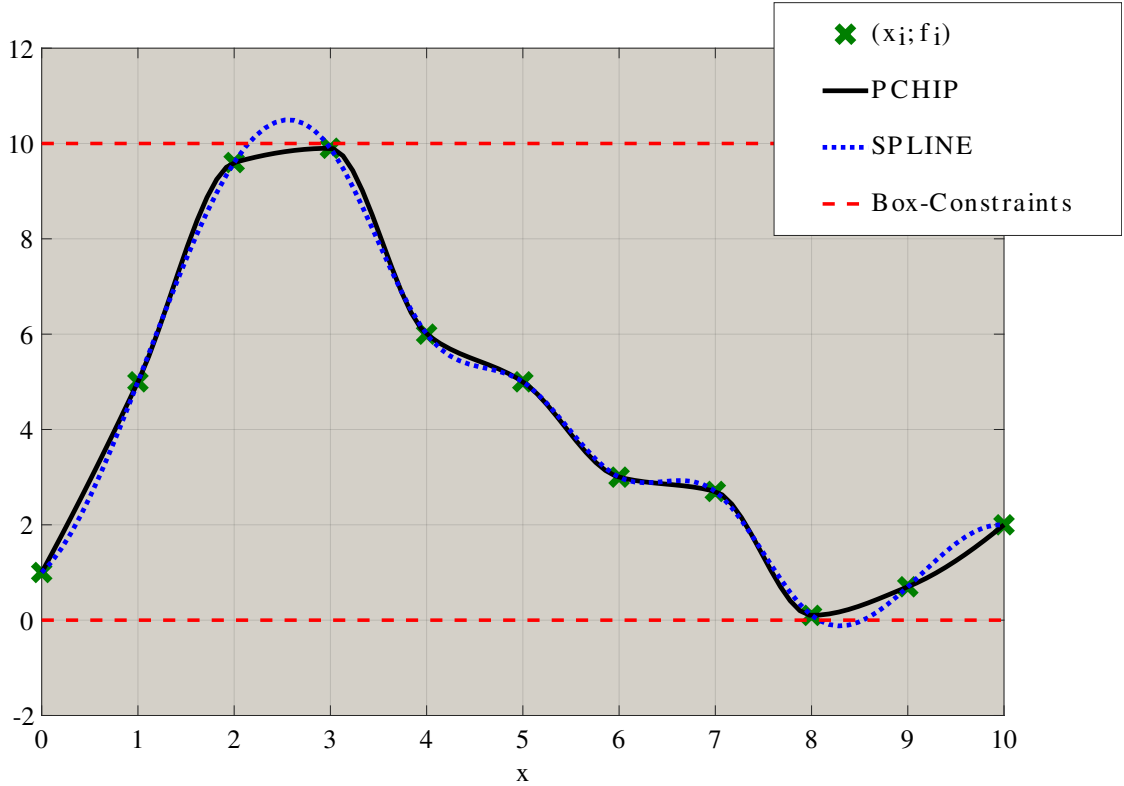
    \captionof{figure}{Example of data points $(x_i,f_i)$ with $x_i=i-1$ and $f_i\in [0,10]$ for $i=1,\dots,11$,\\ comparison PCHIP vs. cubic splines}
    \label{fig:pchipVsSpline}
\end{figure}

\subsection{Approximation property}\label{pchipsApproximationseigenschaft}

Let $f(x)\in \mathcal{C}^1([a,b])$. In the following we present a sketch of the proof to show that 
\begin{align}
\forall \varepsilon>0 \ \exists n\in \mathbb{N}: \ \max\limits_{x\in [a,b]}|p(x)-f(x)|< \varepsilon,
\end{align}
where $p(x)$ is the PCHIP interpolant to the equidistant partition
\begin{align}
\label{partitionAnhang}
\pi_n:a=x_1< x_2< \ldots < x_n=b,
\end{align}
which fulfills
\begin{align*}
p(x_j)=f(x_j)
\end{align*}
for $j=1,\dots,n$. \vspace{5mm}\\The idea is based on nested partitions. Let $\varepsilon>0$ be fixed. For $i=1,2,\ldots$ let
\begin{align*}
h^{(i)}=2^{-i}(b-a)
\end{align*} be the step size of the $i$-th partition
\begin{align*}
\pi^{(i)}:= \{x_j^{(i)}=a+(j-1)h^{(i)}: j=1,\ldots,2^i+1\},
\end{align*}
such that $\pi^{(i)}\subset \pi^{(i+1)}$, i.e. each subsequent partition always refines the previous one. The corresponding PCHIP interpolants for the $i$-th partition are given by $p^{(i)}(x)\in \mathcal{C}^1([a,b])$, such that
\begin{align*}
p^{(i)}(x_j^{(i)})=f(x_j^{(i)})
\end{align*}
for all $j=1,\ldots,2^i+1$. The set 
\begin{align*}
X^{(i)}=\{x\in [a,b]: |p^{(i)}(x)-f(x)|\geq \varepsilon\}
\end{align*}
includes all points $x\in [a,b]$ where the PCHIP interpolant $p^{(i)}(x)$ approximates the function $f(x)$ insufficiently.
 Because of 
\begin{align*}
\lim\limits_{i\to \infty}h^{(i)}=0
\end{align*}
and the monotonicity of the PCHIP interpolants $p^{(i)}(x)$ with respect to the function values $f(x_j^{(i)})$ for all $j=1,\ldots,n$, we conclude that
\begin{align*}
\lim\limits_{i\to\infty} \min\limits_{j}|x_j^{(i)}-x|=0
\end{align*}
and
\begin{align*}
\lim\limits_{i\to \infty} p^{(i)}(x) =f(x)
\end{align*}
holds for all $x\in [a,b]$. This implies the existence of an index $i_* \in \mathbb{N}$ with 
\begin{align*}
\max\limits_{x\in [a,b]}|p^{(i_*)}(x)-f(x)|<\varepsilon
\end{align*}
and $X^{(i_*)}=\{\}$. Therefore, a PCHIP interpolant $p(x)$ exists for all $\varepsilon>0$ for the partition \eqref{partitionAnhang} with $n=2^{i_*}+1$ knots and function values $f(x_j)$ for $j=1,\ldots,n$. Specifically we have
\begin{align*}
x_j=a+(j-1)\frac{b-a}{n-1}.
\end{align*}

A numerical example for determining the number $n$ of partition points corresponding to $\varepsilon=0.2$ is illustrated in Figure \ref{fig:pchipsApproximation}: 

Successively increasing $i=1$ by $i:=i+1$ yields $i_*=4$ with $X^{(i_*)}=\{\}$, i.e. $n=2^{i_*}+1=17$ partition points for the PCHIP interpolant $p(x)$.
The example can obviously be extended for any $\varepsilon>0$. So the approximation accuracy can be controlled by adjusting the number $n$ of partition points.\\

\begin{center}
        \def\svgwidth{\textwidth}
\begingroup%
  \makeatletter%
  \providecommand\color[2][]{%
    \errmessage{(Inkscape) Color is used for the text in Inkscape, but the package 'color.sty' is not loaded}%
    \renewcommand\color[2][]{}%
  }%
  \providecommand\transparent[1]{%
    \errmessage{(Inkscape) Transparency is used (non-zero) for the text in Inkscape, but the package 'transparent.sty' is not loaded}%
    \renewcommand\transparent[1]{}%
  }%
  \providecommand\rotatebox[2]{#2}%
  \newcommand*\fsize{\dimexpr\f@size pt\relax}%
  \newcommand*\lineheight[1]{\fontsize{\fsize}{#1\fsize}\selectfont}%
  \ifx\svgwidth\undefined%
    \setlength{\unitlength}{1195.5155489bp}%
    \ifx\svgscale\undefined%
      \relax%
    \else%
      \setlength{\unitlength}{\unitlength * \real{\svgscale}}%
    \fi%
  \else%
    \setlength{\unitlength}{\svgwidth}%
  \fi%
  \global\let\svgwidth\undefined%
  \global\let\svgscale\undefined%
  \makeatother%
  \begin{picture}(1,1.0570342)%
    \lineheight{1}%
    \setlength\tabcolsep{0pt}%
    \put(0,0){\includegraphics[width=\unitlength,page=1]{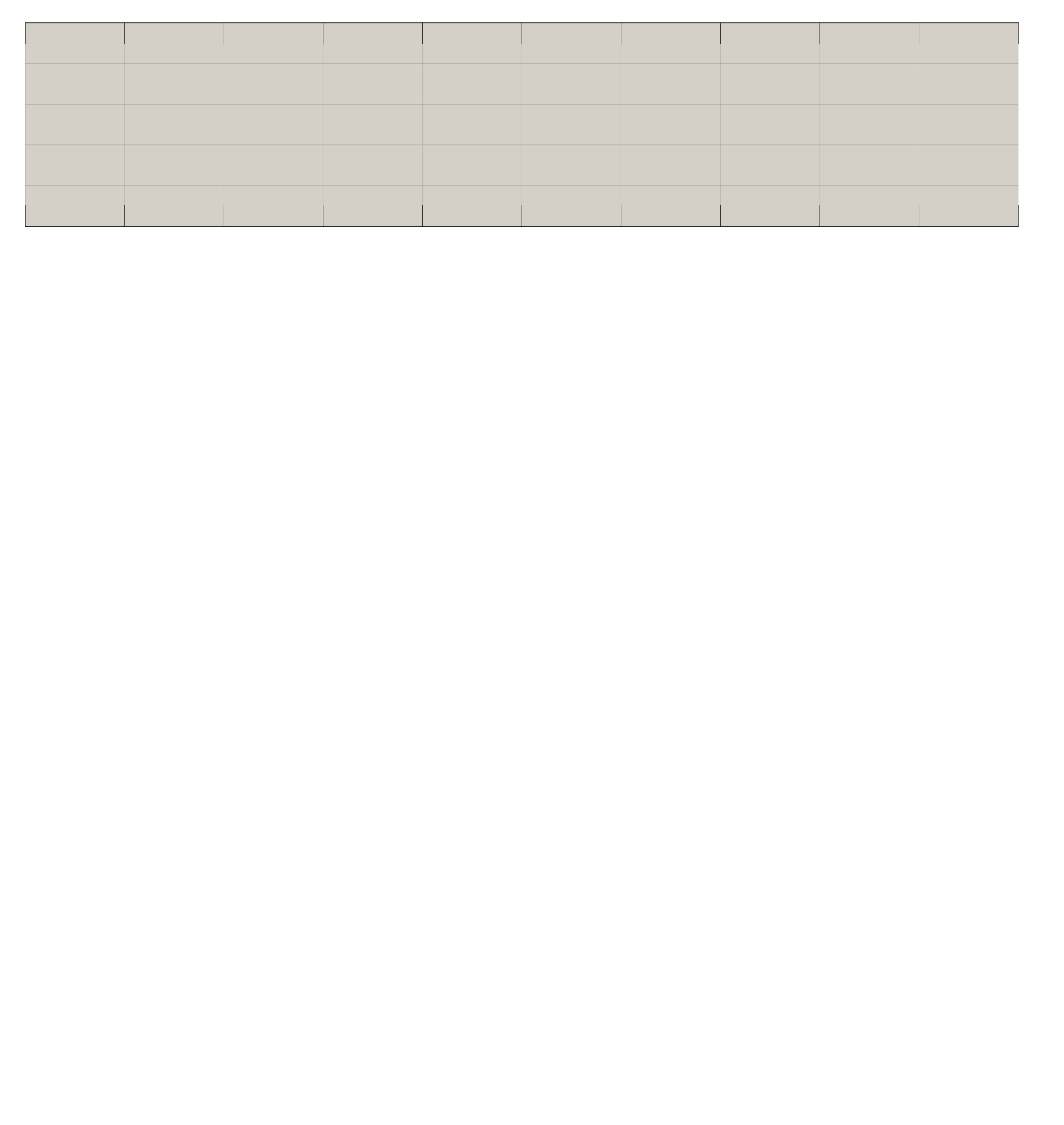}}%
    \put(0.02129247,0.80230442){\makebox(0,0)[lt]{\lineheight{1.25}\smash{\begin{tabular}[t]{l}0\end{tabular}}}}%
    \put(0.11464132,0.80230442){\makebox(0,0)[lt]{\lineheight{1.25}\smash{\begin{tabular}[t]{l}1\end{tabular}}}}%
    \put(0.20799017,0.80230442){\makebox(0,0)[lt]{\lineheight{1.25}\smash{\begin{tabular}[t]{l}2\end{tabular}}}}%
    \put(0.30133902,0.80230442){\makebox(0,0)[lt]{\lineheight{1.25}\smash{\begin{tabular}[t]{l}3\end{tabular}}}}%
    \put(0.39468786,0.80230442){\makebox(0,0)[lt]{\lineheight{1.25}\smash{\begin{tabular}[t]{l}4\end{tabular}}}}%
    \put(0.48803671,0.80230442){\makebox(0,0)[lt]{\lineheight{1.25}\smash{\begin{tabular}[t]{l}5\end{tabular}}}}%
    \put(0.58138556,0.80230442){\makebox(0,0)[lt]{\lineheight{1.25}\smash{\begin{tabular}[t]{l}6\end{tabular}}}}%
    \put(0.6747344,0.80230442){\makebox(0,0)[lt]{\lineheight{1.25}\smash{\begin{tabular}[t]{l}7\end{tabular}}}}%
    \put(0.76808325,0.80230442){\makebox(0,0)[lt]{\lineheight{1.25}\smash{\begin{tabular}[t]{l}8\end{tabular}}}}%
    \put(0.8614321,0.80230442){\makebox(0,0)[lt]{\lineheight{1.25}\smash{\begin{tabular}[t]{l}9\end{tabular}}}}%
    \put(0.95258524,0.80230442){\makebox(0,0)[lt]{\lineheight{1.25}\smash{\begin{tabular}[t]{l}10\end{tabular}}}}%
    \put(0,0){\includegraphics[width=\unitlength,page=2]{pchipsApproximation.pdf}}%
    \put(-0.00045549,0.83721331){\makebox(0,0)[lt]{\lineheight{1.25}\smash{\begin{tabular}[t]{l}0\end{tabular}}}}%
    \put(-0.00045549,0.87543494){\makebox(0,0)[lt]{\lineheight{1.25}\smash{\begin{tabular}[t]{l}2\end{tabular}}}}%
    \put(-0.00045549,0.91365658){\makebox(0,0)[lt]{\lineheight{1.25}\smash{\begin{tabular}[t]{l}4\end{tabular}}}}%
    \put(-0.00045549,0.95187821){\makebox(0,0)[lt]{\lineheight{1.25}\smash{\begin{tabular}[t]{l}6\end{tabular}}}}%
    \put(-0.00045549,0.99009984){\makebox(0,0)[lt]{\lineheight{1.25}\smash{\begin{tabular}[t]{l}8\end{tabular}}}}%
    \put(-0.0048469,1.02832148){\makebox(0,0)[lt]{\lineheight{1.25}\smash{\begin{tabular}[t]{l}10\end{tabular}}}}%
    \put(0,0){\includegraphics[width=\unitlength,page=3]{pchipsApproximation.pdf}}%
    \put(0.02129247,0.53789449){\makebox(0,0)[lt]{\lineheight{1.25}\smash{\begin{tabular}[t]{l}0\end{tabular}}}}%
    \put(0.11464132,0.53789449){\makebox(0,0)[lt]{\lineheight{1.25}\smash{\begin{tabular}[t]{l}1\end{tabular}}}}%
    \put(0.20799017,0.53789449){\makebox(0,0)[lt]{\lineheight{1.25}\smash{\begin{tabular}[t]{l}2\end{tabular}}}}%
    \put(0.30133902,0.53789449){\makebox(0,0)[lt]{\lineheight{1.25}\smash{\begin{tabular}[t]{l}3\end{tabular}}}}%
    \put(0.39468786,0.53789449){\makebox(0,0)[lt]{\lineheight{1.25}\smash{\begin{tabular}[t]{l}4\end{tabular}}}}%
    \put(0.48803671,0.53789449){\makebox(0,0)[lt]{\lineheight{1.25}\smash{\begin{tabular}[t]{l}5\end{tabular}}}}%
    \put(0.58138556,0.53789449){\makebox(0,0)[lt]{\lineheight{1.25}\smash{\begin{tabular}[t]{l}6\end{tabular}}}}%
    \put(0.6747344,0.53789449){\makebox(0,0)[lt]{\lineheight{1.25}\smash{\begin{tabular}[t]{l}7\end{tabular}}}}%
    \put(0.76808325,0.53789449){\makebox(0,0)[lt]{\lineheight{1.25}\smash{\begin{tabular}[t]{l}8\end{tabular}}}}%
    \put(0.8614321,0.53789449){\makebox(0,0)[lt]{\lineheight{1.25}\smash{\begin{tabular}[t]{l}9\end{tabular}}}}%
    \put(0.95258524,0.53789449){\makebox(0,0)[lt]{\lineheight{1.25}\smash{\begin{tabular}[t]{l}10\end{tabular}}}}%
    \put(0,0){\includegraphics[width=\unitlength,page=4]{pchipsApproximation.pdf}}%
    \put(-0.00045549,0.57280339){\makebox(0,0)[lt]{\lineheight{1.25}\smash{\begin{tabular}[t]{l}0\end{tabular}}}}%
    \put(-0.00045549,0.61102502){\makebox(0,0)[lt]{\lineheight{1.25}\smash{\begin{tabular}[t]{l}2\end{tabular}}}}%
    \put(-0.00045549,0.64924665){\makebox(0,0)[lt]{\lineheight{1.25}\smash{\begin{tabular}[t]{l}4\end{tabular}}}}%
    \put(-0.00045549,0.68746828){\makebox(0,0)[lt]{\lineheight{1.25}\smash{\begin{tabular}[t]{l}6\end{tabular}}}}%
    \put(-0.00045549,0.72568992){\makebox(0,0)[lt]{\lineheight{1.25}\smash{\begin{tabular}[t]{l}8\end{tabular}}}}%
    \put(-0.0048469,0.76391155){\makebox(0,0)[lt]{\lineheight{1.25}\smash{\begin{tabular}[t]{l}10\end{tabular}}}}%
    \put(0,0){\includegraphics[width=\unitlength,page=5]{pchipsApproximation.pdf}}%
    \put(0.02129247,0.27348457){\makebox(0,0)[lt]{\lineheight{1.25}\smash{\begin{tabular}[t]{l}0\end{tabular}}}}%
    \put(0.11464132,0.27348457){\makebox(0,0)[lt]{\lineheight{1.25}\smash{\begin{tabular}[t]{l}1\end{tabular}}}}%
    \put(0.20799017,0.27348457){\makebox(0,0)[lt]{\lineheight{1.25}\smash{\begin{tabular}[t]{l}2\end{tabular}}}}%
    \put(0.30133902,0.27348457){\makebox(0,0)[lt]{\lineheight{1.25}\smash{\begin{tabular}[t]{l}3\end{tabular}}}}%
    \put(0.39468786,0.27348457){\makebox(0,0)[lt]{\lineheight{1.25}\smash{\begin{tabular}[t]{l}4\end{tabular}}}}%
    \put(0.48803671,0.27348457){\makebox(0,0)[lt]{\lineheight{1.25}\smash{\begin{tabular}[t]{l}5\end{tabular}}}}%
    \put(0.58138556,0.27348457){\makebox(0,0)[lt]{\lineheight{1.25}\smash{\begin{tabular}[t]{l}6\end{tabular}}}}%
    \put(0.6747344,0.27348457){\makebox(0,0)[lt]{\lineheight{1.25}\smash{\begin{tabular}[t]{l}7\end{tabular}}}}%
    \put(0.76808325,0.27348457){\makebox(0,0)[lt]{\lineheight{1.25}\smash{\begin{tabular}[t]{l}8\end{tabular}}}}%
    \put(0.8614321,0.27348457){\makebox(0,0)[lt]{\lineheight{1.25}\smash{\begin{tabular}[t]{l}9\end{tabular}}}}%
    \put(0.95258524,0.27348457){\makebox(0,0)[lt]{\lineheight{1.25}\smash{\begin{tabular}[t]{l}10\end{tabular}}}}%
    \put(0,0){\includegraphics[width=\unitlength,page=6]{pchipsApproximation.pdf}}%
    \put(-0.00045549,0.30839346){\makebox(0,0)[lt]{\lineheight{1.25}\smash{\begin{tabular}[t]{l}0\end{tabular}}}}%
    \put(-0.00045549,0.3463533){\makebox(0,0)[lt]{\lineheight{1.25}\smash{\begin{tabular}[t]{l}2\end{tabular}}}}%
    \put(-0.00045549,0.38431314){\makebox(0,0)[lt]{\lineheight{1.25}\smash{\begin{tabular}[t]{l}4\end{tabular}}}}%
    \put(-0.00045549,0.42227298){\makebox(0,0)[lt]{\lineheight{1.25}\smash{\begin{tabular}[t]{l}6\end{tabular}}}}%
    \put(-0.00045549,0.46023282){\makebox(0,0)[lt]{\lineheight{1.25}\smash{\begin{tabular}[t]{l}8\end{tabular}}}}%
    \put(-0.0048469,0.49819266){\makebox(0,0)[lt]{\lineheight{1.25}\smash{\begin{tabular}[t]{l}10\end{tabular}}}}%
    \put(0,0){\includegraphics[width=\unitlength,page=7]{pchipsApproximation.pdf}}%
    \put(0.02129247,0.00907464){\makebox(0,0)[lt]{\lineheight{1.25}\smash{\begin{tabular}[t]{l}0\end{tabular}}}}%
    \put(0.11464132,0.00907464){\makebox(0,0)[lt]{\lineheight{1.25}\smash{\begin{tabular}[t]{l}1\end{tabular}}}}%
    \put(0.20799017,0.00907464){\makebox(0,0)[lt]{\lineheight{1.25}\smash{\begin{tabular}[t]{l}2\end{tabular}}}}%
    \put(0.30133902,0.00907464){\makebox(0,0)[lt]{\lineheight{1.25}\smash{\begin{tabular}[t]{l}3\end{tabular}}}}%
    \put(0.39468786,0.00907464){\makebox(0,0)[lt]{\lineheight{1.25}\smash{\begin{tabular}[t]{l}4\end{tabular}}}}%
    \put(0.48803671,0.00907464){\makebox(0,0)[lt]{\lineheight{1.25}\smash{\begin{tabular}[t]{l}5\end{tabular}}}}%
    \put(0.58138556,0.00907464){\makebox(0,0)[lt]{\lineheight{1.25}\smash{\begin{tabular}[t]{l}6\end{tabular}}}}%
    \put(0.6747344,0.00907464){\makebox(0,0)[lt]{\lineheight{1.25}\smash{\begin{tabular}[t]{l}7\end{tabular}}}}%
    \put(0.76808325,0.00907464){\makebox(0,0)[lt]{\lineheight{1.25}\smash{\begin{tabular}[t]{l}8\end{tabular}}}}%
    \put(0.8614321,0.00907464){\makebox(0,0)[lt]{\lineheight{1.25}\smash{\begin{tabular}[t]{l}9\end{tabular}}}}%
    \put(0.95258524,0.00907464){\makebox(0,0)[lt]{\lineheight{1.25}\smash{\begin{tabular}[t]{l}10\end{tabular}}}}%
    \put(0,0){\includegraphics[width=\unitlength,page=8]{pchipsApproximation.pdf}}%
    \put(-0.00045549,0.04398353){\makebox(0,0)[lt]{\lineheight{1.25}\smash{\begin{tabular}[t]{l}0\end{tabular}}}}%
    \put(-0.00045549,0.08194337){\makebox(0,0)[lt]{\lineheight{1.25}\smash{\begin{tabular}[t]{l}2\end{tabular}}}}%
    \put(-0.00045549,0.11990321){\makebox(0,0)[lt]{\lineheight{1.25}\smash{\begin{tabular}[t]{l}4\end{tabular}}}}%
    \put(-0.00045549,0.15786305){\makebox(0,0)[lt]{\lineheight{1.25}\smash{\begin{tabular}[t]{l}6\end{tabular}}}}%
    \put(-0.00045549,0.1958229){\makebox(0,0)[lt]{\lineheight{1.25}\smash{\begin{tabular}[t]{l}8\end{tabular}}}}%
    \put(-0.0048469,0.23378274){\makebox(0,0)[lt]{\lineheight{1.25}\smash{\begin{tabular}[t]{l}10\end{tabular}}}}%
    \put(0,0){\includegraphics[width=\unitlength,page=9]{pchipsApproximation.pdf}}%
    \put(0.88556428,1.01073415){\color[rgb]{0,0,0}\makebox(0,0)[lt]{\lineheight{1.25}\smash{\begin{tabular}[t]{l}$i=1$\end{tabular}}}}%
    \put(0.8863624,0.74482254){\color[rgb]{0,0,0}\makebox(0,0)[lt]{\lineheight{1.25}\smash{\begin{tabular}[t]{l}$i=2$\end{tabular}}}}%
    \put(0.88576667,0.48142442){\color[rgb]{0,0,0}\makebox(0,0)[lt]{\lineheight{1.25}\smash{\begin{tabular}[t]{l}$i=3$\end{tabular}}}}%
    \put(0.88138879,0.21519495){\color[rgb]{0,0,0}\makebox(0,0)[lt]{\lineheight{1.25}\smash{\begin{tabular}[t]{l}$i_*=4$\end{tabular}}}}%
  \end{picture}%
\endgroup%

    
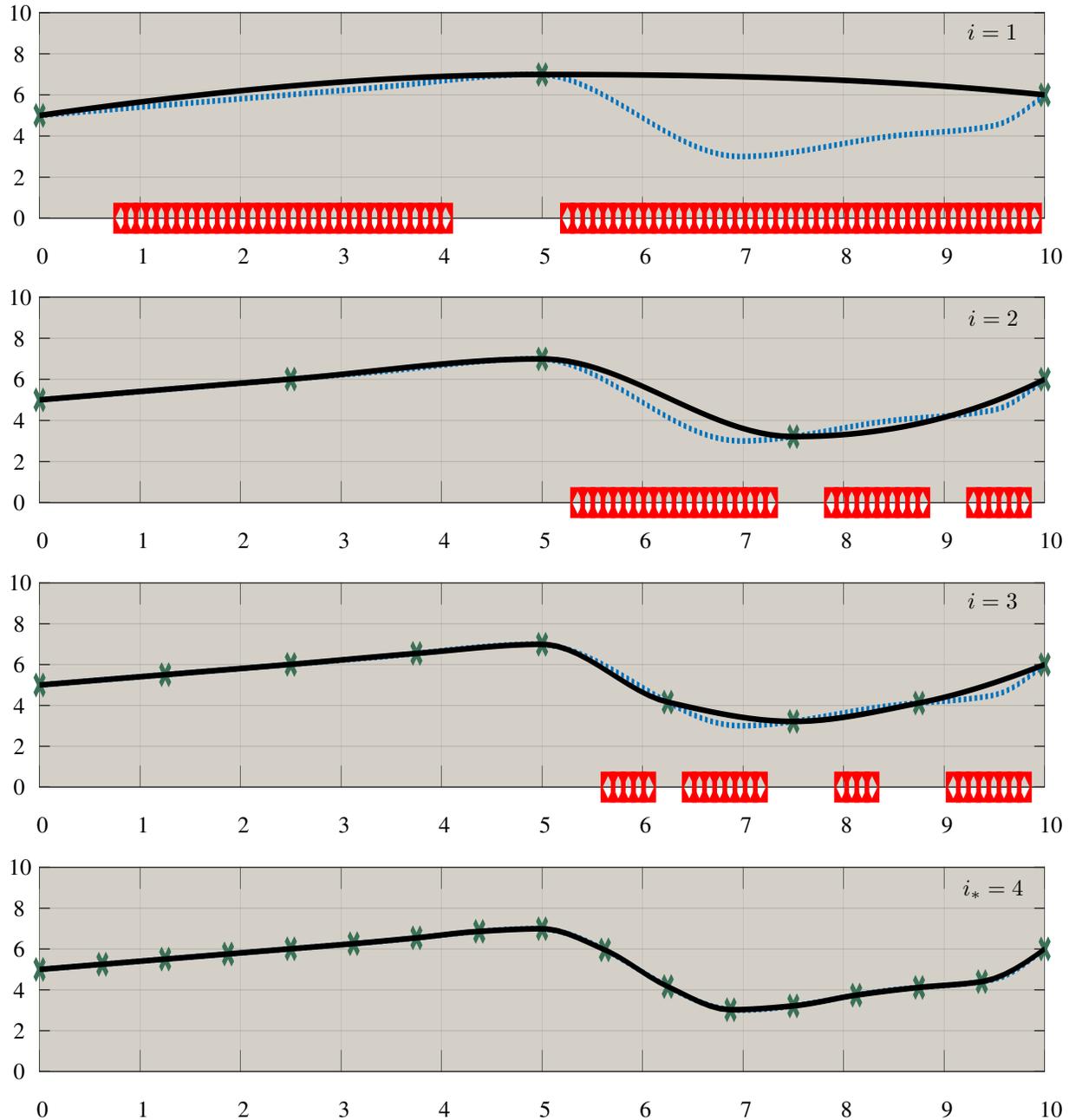
\captionof{figure}{$f(x)$ (dotted blue) vs. PCHIP interpolants $p^{(i)}(x)$ (solid black) for $i=1,\dots,4$;\\ red regions correspond to $X^{(i)}$ for $\varepsilon=0.2$}
    \label{fig:pchipsApproximation}
\end{center}

\subsection{Calculation of the gradient}\label{pchipsGradientenberechnung}

Finally we look at the sensitivity of $p(x)$ with respect to $f_i$ for all $i=1,2,\ldots, n$, that means we compute $\frac{\partial p}{\partial f_i}(x)$ for all $i=1,\ldots, n$. Here we assume without loss of generality that $\sgn(\Delta_k \Delta_{k+1})>0$ holds true. (The derivations are simplified if $d_i=0$ for some $i=1,\ldots,n$.)
Since $p(x)$ is a piecewise defined function, we consider the derivatives for certain subintervals $I_i=[x_i, x_{i+1}]$. 
In the following, the individual derivatives are listed, whereby we first consider the special cases depending on $d_1$, i.e. $\frac{\partial p}{\partial f_i}(x)$ for $i=1,\ 2,\ 3$:

\begin{align*}
\frac{\partial p}{\partial f_1}(x)= \left\{\begin{array}{ll} H_1^1(x)-\frac{3}{2h} H_3^1(x)-\frac{2}{h}\frac{\Delta_2^2}{(\Delta_1+\Delta_2)^2} H_4^1(x), & x\in [x_1,x_2], \\
         -\frac{2}{h} \frac{\Delta_2^2}{(\Delta_1+\Delta_2)^2} H_3^2(x), & x\in [x_2,x_3],\\
         0, & \text{else}.
          \end{array}\right. 
\end{align*}

\begin{align*}
\frac{\partial p}{\partial f_2}(x)= \left\{\begin{array}{ll} H_2^1(x)+ \frac{2}{h}H_3^1(x)+\frac{2}{h}\frac{\Delta_2-\Delta_1}{\Delta_1+\Delta_2}H_4^1(x), & x\in [x_1,x_2], \\
H_1^2(x)-\frac{2}{h} \frac{\Delta_2-\Delta_1}{\Delta_2+\Delta_3} H_3^2(x)-\frac{2}{h}\frac{\Delta_3^2}{(\Delta_2+\Delta_3)^2} H_4^2(x), & x\in [x_2,x_3], \\
        -\frac{2}{h}\frac{\Delta_3^2}{(\Delta_2+\Delta_3)^2}H_3^3(x) , & x\in [x_3,x_4],\\
         0, & \text{else}.
          \end{array}\right. 
\end{align*}

\begin{align*}
\frac{\partial p}{\partial f_3}(x)= \left\{\begin{array}{ll}\frac{1}{2h}H_3^1(x)+ \frac{2}{h} \frac{\Delta_1^2}{(\Delta_1+\Delta_2)^2} H_4^1(x), & x\in [x_1,x_2], \\
H_2^2(x)+\frac{2}{h}\frac{\Delta_1^2}{(\Delta_1+\Delta_2)^2}H_3^2(x)+\frac{2}{h}\frac{\Delta_3-\Delta_2}{\Delta_2+\Delta_3}H_4^2(x), & x\in [x_2,x_3], \\        
     H_1^3(x)+\frac{2}{h}\frac{\Delta_3-\Delta_2}{\Delta_2+\Delta_3}H_3^3(x)-\frac{2}{h}\frac{\Delta_4^2}{(\Delta_3+\Delta_4)^2}H_4^3(x) , & x\in [x_3,x_4],\\
    -\frac{2}{h}\frac{\Delta_4^2}{(\Delta_3+\Delta_4)^2}H_3^4(x), & x\in [x_4,x_5],\\
              0, & \text{else}.
          \end{array}\right. 
\end{align*}

For $3<i<n-2$:

\begin{align*}
\frac{\partial p}{\partial f_i}(x)= \left\{\begin{array}{ll} 
\frac{2}{h}\frac{\Delta_{i-2}^2}{(\Delta_{i-1}+\Delta_{i-2})^2}H_4^{i-2}(x), & x\in [x_{i-2},x_{i-1}], \\   
H_2^{i-1}(x)+\frac{2}{h}\frac{\Delta_{i-2}^2}{(\Delta_{i-1}+\Delta_{i-2})^2} H_3^{i-1}(x)+\frac{2}{h}\frac{\Delta_i-\Delta_{i-1}}{\Delta_{i-1}+\Delta_i}H_4^{i-1}(x), & x\in [x_{i-1},x_{i}], \\
H_1^i(x)+\frac{2}{h}\frac{\Delta_i-\Delta_{i-1}}{\Delta_{i-1}+\Delta_i}H_3^i(x)-\frac{2}{h}\frac{\Delta_{i+1}^2}{(\Delta_i+\Delta_{i+1})^2}H_4^i(x), & x\in [x_i,x_{i+1}], \\
-\frac{2}{h}\frac{\Delta_{i+1}^2}{(\Delta_i+\Delta_{i+1})^2}H_3^{i+1}(x), & x\in [x_{i+1},x_{i+2}], \\        
              0, & \text{else}.
          \end{array}\right. 
\end{align*}
Other special cases (regarding $d_n$):
\begin{align*}
\frac{\partial p}{\partial f_{n-2}}(x)= \left\{\begin{array}{ll} 
\frac{2}{h}\frac{\Delta_{n-4}^2}{(\Delta_{n-3}+\Delta_{n-4})^2}H_4^{n-4}(x), & x\in [x_{n-4},x_{n-3}], \\
H_2^{n-3}(x)+\frac{2}{h}\frac{\Delta_{n-4}^2}{(\Delta_{n-3}+\Delta_{n-4})^2}H_3^{n-3}(x)+\frac{2}{h}\frac{\Delta_{n-2}-\Delta_{n-3}}{\Delta_{n-3}+\Delta_{n-2}}H_4^{n-3}(x), & x\in [x_{n-3},x_{n-2}], \\
H_1^{n-2}(x)+\frac{2}{h}\frac{\Delta_{n-2}+\Delta_{n-3}}{\Delta_{n-3}+\Delta_{n-2}}H_3^{n-2}(x)-\frac{2}{h}\frac{\Delta_{n-1}^2}{(\Delta_{n-2}+\Delta_{n-1})^2}H_4^{n-2}(x), & x\in [x_{n-2},x_{n-1},] \\
-\frac{2}{h}\frac{\Delta_{n-1}^2}{(\Delta_{n-2}+\Delta_{n-1})^2}H_3^{n-1}(x)+\frac{1}{2h}H_4^{n-1}(x), & x\in [x_{n-1},x_{n}], \\        
              0, & \text{else}.
          \end{array}\right. 
\end{align*}

\begin{align*}
\frac{\partial p}{\partial f_{n-1}}(x)= \left\{\begin{array}{ll} 
\frac{2}{h}\frac{\Delta_{n-3}^2}{(\Delta_{n-2}+\Delta_{n-3})^2}H_4^{n-3}(x), & x\in [x_{n-3},x_{n-2}], \\
H_2^{n-2}(x)+\frac{2}{h}\frac{\Delta_{n-3}^2}{(\Delta_{n-2}+\Delta_{n-3})^2}H_3^{n-2}(x)+\frac{2}{h}\frac{\Delta_{n-1}+\Delta_{n-2}}{\Delta_{n-2}+\Delta_{n-1}}H_4^{n-2}(x), & x\in [x_{n-2},x_{n-1}], \\
H_1^{n-1}(x)+\frac{2}{h}\frac{\Delta_{n-1}-\Delta_{n-2}}{\Delta_{n-2}+\Delta_{n-1}}H_3^{n-1}(x)-\frac{2}{h}H_4^{n-1}(x), & x\in [x_{n-1},x_{n}], \\        
              0, & \text{else}.
          \end{array}\right. 
\end{align*}

\begin{align*}
\frac{\partial p}{\partial f_n}(x)= \left\{\begin{array}{ll} 
\frac{2}{h}\frac{\Delta_{n-2}^2}{(\Delta_{n-1}+\Delta_{n-2})^2}H_4^{n-2}(x), & x\in [x_{n-2},x_{n-1}], \\
H_2^{n-1}(x)+\frac{2}{h}\frac{\Delta_{n-2}^2}{(\Delta_{n-1}+\Delta_{n-2})^2}H_3^{n-1}(x)+\frac{3}{2h}H_4^{n-1}(x), & x\in [x_{n-1},x_{n}], \\        
              0, & \text{else}.
          \end{array}\right. 
\end{align*}
\end{appendix}

\bibliography{references} 

\end{document}